\documentclass[11pt]{article}

\usepackage[margin=1in]{geometry}

\usepackage[]{amsmath,amssymb,epsfig}
\usepackage{amsthm}
\usepackage{amsmath}
\usepackage{amssymb}
\usepackage{graphicx}
\usepackage{epstopdf}
\usepackage{comment}
\usepackage{array}
\usepackage{algorithm}
\usepackage{url}
\usepackage[numbers]{natbib}

\usepackage{lscape}
\usepackage{algpseudocode}
\usepackage{setspace}
\usepackage{multicol}
\usepackage{multirow}
\usepackage{color}
\usepackage{colortbl}
\usepackage{xcolor}
\usepackage{esvect}

\usepackage{placeins}

\usepackage[%
    font={small,sf},
    labelfont=bf,
    format=hang,    
    format=plain,
    margin=0pt,
    width=0.8\textwidth,
]{caption}
\usepackage[list=true]{subcaption}

\newtheorem{theorem}{Theorem}

\newtheorem{assumption}{Assumption}

\newtheorem{claim}{Claim}

\newtheorem{corollary}{Corollary}

\newtheorem{lemma}{Lemma}

\newtheorem{proposition}{Proposition}
\newtheorem{remark}{Remark}

\numberwithin{equation}{section}
\newtheorem{thmx}{Theorem}

	\renewcommand{\thefootnote}{\arabic{footnote}}

\DeclareMathOperator{\Tr}{Tr}

\DeclareMathOperator{\diag}{diag}

\DeclareMathOperator{\spn}{span}
\newcommand{\calA}{\ensuremath{\mathcal{A}}}
\newcommand{\calB}{\ensuremath{\mathcal{B}}}

\newcommand{\calH}{\ensuremath{\mathcal{H}}}

\newcommand{\calG}{\ensuremath{\mathcal{G}}}
\newcommand{\calI}{\ensuremath{\mathcal{I}}}

\newcommand{\calN}{\ensuremath{\mathcal{N}}}
\newcommand{\calT}{\ensuremath{\mathcal{T}}}

\newcommand{\calV}{\ensuremath{\mathcal{V}}}

\newcommand{\calE}{\ensuremath{\mathcal{E}}}
\newcommand{\calL}{\ensuremath{\mathcal{L}}}

\newcommand{\norm}[1]{\left\|{#1}\right\|}
\newcommand{\abs}[1]{\left|{#1}\right|}

\newcommand{\set}[1]{\left\{{#1}\right\}}

\newcommand{\est}[1]{\hat{#1}}
\newcommand{\expec}{\ensuremath{\mathbb{E}}}
\newcommand{\matR}{\ensuremath{\mathbb{R}}}

\newcommand{\argmin}[1]{\underset{#1}{\operatorname{argmin}}}

\newcommand{\prob}{\ensuremath{\mathbb{P}}}

\newcommand{\logit}{\operatorname{logit}}

\newcommand{\red}[1]{\textcolor{black}{#1}}

\newcommand{\ones}{\ensuremath{\mathbf{1}}} 

\newcommand{\complincmat}{\ensuremath{C}}

\newcommand{\Ldag}{L^\dagger(\lambda)}

\newcommand{\Lcom}{L_{\operatorname{com}}}

\newcommand{\lamin}{\lambda_{\operatorname{min}}}

\newcommand{\okap}{\overline{\kappa}}

\newcommand{\lfreqsp}{\calV_{\tau}}
\newcommand{\hfreqsp}{\calV^{\perp}_{\tau}}
\newcommand{\lfreqproj}{P_{\calV_{\tau}}}
\newcommand{\hfreqproj}{P_{\calV^{\perp}_{\tau}}}

\newcommand{\rev}[1]{\textcolor{black}{#1}}
\begin{document}
\title{Dynamic Ranking and Translation Synchronization}

\author{Ernesto Araya\footnotemark[1]\\ \texttt{ernesto-javier.araya-valdivia@inria.fr} \and Eglantine Karl\'e \footnotemark[1]\\ \texttt{eglantine.karle@inria.fr} \and Hemant Tyagi \footnotemark[1]\\ \texttt{hemant.tyagi@inria.fr}}

\renewcommand{\thefootnote}{\fnsymbol{footnote}}
\footnotetext[1]{Inria, Univ. Lille, CNRS, UMR 8524 - Laboratoire Paul Painlev\'{e}, F-59000 \\
\indent Authors are written in alphabetical order.}
\renewcommand{\thefootnote}{\arabic{footnote}}

\maketitle

\begin{abstract}
In many applications, such as sport tournaments or recommendation systems, we have at our disposal data consisting of pairwise comparisons between a set of $n$ items (or players). The objective is to use this data to infer the latent strength of each item and/or their ranking. Existing results for this problem predominantly focus on the setting consisting of a single comparison graph $G$. However, there exist scenarios (e.g., sports tournaments) where the pairwise comparison data evolves with time. Theoretical results for this dynamic setting are relatively limited, and \rev{are} the focus of this paper.

We study an extension of the \emph{translation synchronization} problem, to the dynamic setting. In this setup, we are given a sequence of comparison graphs $(G_t)_{t\in \calT}$, where $\calT \subset [0,1]$ is a grid representing the time domain, and for each item $i$ and time $t\in \calT$ there is an associated unknown strength parameter $z^*_{t,i}\in \matR$. We aim to recover, for $t\in\calT$, the strength vector $z^*_t=(z^*_{t,1},\dots,z^*_{t,n})$ from noisy measurements of $z^*_{t,i}-z^*_{t,j}$, where $\set{i,j}$ is an edge in $G_t$. Assuming that $z^*_t$ evolves \emph{smoothly} in $t$, we propose two estimators -- one based on a smoothness-penalized least squares approach and the other based on projection onto the low frequency eigenspace of a suitable smoothness operator. For both estimators, we provide finite sample bounds for the $\ell_2$ estimation error under the assumption that $G_t$ is connected for all $t\in \calT$, thus proving the consistency of the proposed methods in terms of the grid size $|\calT|$.
We complement our theoretical findings with experiments on synthetic and real data.

\end{abstract}

%
\section{Introduction} \label{sec:intro}
Ranking problems are ubiquitous in statistics and have a wide range of applications, e.g.,  recommendation systems \cite{jannach2016recommender}, sports and gaming tournaments \cite{cattelan2013dynamic,bong2020nonparametric}, biological systems \cite{carpenter2006cellprofiler,shah2017simple} and information retrieval \cite{cao2007infoRet}. In these problems, we are given a set of items (or players) and data that typically consists of outcomes of pairwise comparisons between the items.  For instance, in movie recommender systems,
the typical data sets contain information about the users' preference for one movie over another, by means of a rating score. In sports tournaments, we are given win/loss information between pairs of teams. The aim then is to recover a ranking of the items given a collection of such pairwise comparison outcomes. This is often achieved by
estimating a (latent) score for each item, which measures its inherent strength or quality. 

From the modeling viewpoint, it is natural to represent the set of pairs of items to be compared as a simple undirected graph $G = ([n], \calE)$ where each item is assigned a label in $\{1,\dots, n\}$ and an edge $\{i, j\} \in \calE$
exists iff $i$ and $j$ are compared in the data generating process. In sports and gaming applications, for example, we can think of $G$ as the tournament design, representing the information of who plays who. A necessary condition for recovering the underlying ranking is that the graph $G$ is connected. Otherwise, there
will exist at least two connected components of $G$ and it is then impossible to produce a full ranking since no comparison information is available for items 
in one component with respect to the other.

Let us now describe a relatively simple yet popular model for ranking which is at the heart of what we study in this paper. Denote $z^*=(z^*_1,\dots,z^*_n)^\top \in \matR^n$ to be an unknown vector representing the latent strengths of the items. For each $\set{i,j} \in \calE$, we are given a noisy measurement  
\begin{equation}\label{eq:intro_tran_sync}
    y_{ij}=z^*_i-z^*_j+\epsilon_{ij},
\end{equation}
where $\epsilon$ denotes noise. The goal is to recover the unknown signal $z^*$ given the data $(y_{ij})_{\set{i,j}\in\calE}$. The model in \eqref{eq:intro_tran_sync} is referred to as \emph{translation synchronization} (TranSync) \cite{huang2017translation} in the literature and has received attention in both theory \cite{Cucuringu2016rankSync,ranksync21} and applications. 
It is clear that recovering $z^*$ up to a global shift is the best we can hope for. In addition, in the noiseless case, we can exactly recover $z^*$ if and only if $G$ is connected, by considering the following procedure. Take a spanning tree of $G$ and fix any value for the root node, then traverse the tree while unwrapping the information for each node of tree (by adding the offsets $z^*_i-z^*_j$). 
The following are the two main types of applications of the translation synchronization model. 
\begin{itemize}
    \item \emph{Time synchronization of networks.} This problem arises in engineering in the context of time synchronization of distributed networks \cite{Karp2003clock,Giridhar2006DistributedCS}. Here, we are given noisy measurements of the time offset $z^*_i-z^*_j$, where $z^*_i$ represents the time measured by a clock at node $i$. The goal is to recover the time vector $z^*$ from the noisy measurements. 
    
    \item \emph{Ranking.} Here, $z^*_i$ represents the strength/quality of item $i$, and the goal is to produce a ranking by the estimation of $z^*$ from the data $(y_{ij})_{\set{i,j}\in\calE}$. For instance, in movie rating systems, $y_{ij}$ is formed by taking the average of the score differences of movies attributed by individual users \cite{Gleich2011rankAggreg}. There is an increasing literature on the theoretical side along with applications of translation synchronization to ranking problems \cite{Gleich2011rankAggreg,Cucuringu2016rankSync,ranksync21}. Although the model presented in \eqref{eq:intro_tran_sync} is linear, it has been shown that some nonlinear models, such as the popular Bradley-Terry-Luce (BTL) model, can be approximated in the form \eqref{eq:intro_tran_sync} by considering an appropriate transformation \cite{Hendrickx2019,hendrickx2020minimax} (see Remark \ref{rem:btl_setup}). 
\end{itemize}

\subsection{Dynamic TranSync model}
%

\paragraph{Dynamic setting for ranking.} We will refer to the setup \eqref{eq:intro_tran_sync} as the `static' case since a single comparison graph -- which does not evolve with time -- is considered, and the same is true for the strength parameter $z^*$. However, in many applications, it is reasonable to assume that the latent strengths of the items will change over time and so do the comparison graphs. For instance, in a sport tournament, different teams or players will meet at different times and their overall strength (or abilities) can either improve (by learning new skills, for example) or deteriorate (\rev{due} to injuries, lack of motivation, etc). Similarly, in movie recommendation systems, the evaluation from the general public that a particular film receives is likely to evolve according to time-dependent factors (date of release, change in viewer's taste  etc). The dynamic setting extends the static one by considering a sequence of comparison graphs $G_t=([n],\calE_t)$, where the sub-index $t$ represents the time evolution. Moreover, the latent strength parameter of each item will be a function of $t$. 

\paragraph{Dynamic TranSync model.} 
In this paper, we will focus on a dynamic version of the TranSync model \cite{huang2017translation}. Denoting $\calT$ to be a uniform grid\footnote{\rev{This is only for ease of exposition, the results hold for any discrete set of size $T+1$.}} of $[0,1]$ of size $T+1$, we observe at each time $t \in \calT$   
\begin{equation}\label{eq:transync_intro}
y_{ij}(t) = z^*_{t,i}-z^*_{t,j} + \epsilon_{ij}(t), \quad \set{i,j} \in \calE_t.
\end{equation}
Here $z^*_{t,i},z^*_{t,j}$ are the strength parameters of item $i$ and $j$ respectively at time $t$, and $\epsilon_{ij}(t)$ represents the noise in the data. 
Our objective is to recover the vector $z^*_t = (z^*_{t,1},\dots,z^*_{t,n})^{\top}$ up to an additive shift via the estimate $\est{z}_t$, from the observations $y_{ij}(t)$. In particular, we would like to establish consistency w.r.t $T$ (with $n$ fixed) in the sense that the mean squared error (MSE)
\begin{equation} \label{eq:mse_consis}
\frac{1}{\abs{\calT}} \sum_{t \in \calT} \norm{\est{z}_t - z^*_t}_2^2 \overset{T \rightarrow \infty}{\longrightarrow} 0.
\end{equation}
The following remarks are in order.
\begin{itemize}
\item If no additional assumption is made on the time-evolution of the strengths, then the strategy of applying existing `static' case methods for each $t$ is the best one possible. However in this case, we will typically have $(1/\abs{\calT}) \sum_{t \in \calT} \norm{\est{z}_t - z^*_t}_2^2  = \Omega(1)$
as $T \rightarrow \infty$.  

\item The individual comparison graphs $G_t$ can be very sparse, and possibly disconnected for some (or even all) $t \in \calT$. Thus, using only the information provided by $G_t$ for inferring $z^*_t$ is doomed to fail.
\end{itemize}
Yet in practice, it is reasonable to assume that the strength vector $z^*_t$ evolves \emph{smoothly} with time. This is also the setup considered in this paper and is discussed formally in Section \ref{sec:model}, see Assumption \ref{assump:smooth}. The smoothness constraints help us in the estimation of $z^*_t$ by leveraging pairwise data at different time instants, thus allowing us to establish the consistency condition in \eqref{eq:mse_consis}. 

\subsection{Related work} \label{subsec:relwork}
We now discuss works that are closely related to the problem setup considered in this paper.
\paragraph{Static ranking models.}  The BTL model, introduced in \cite{bradley1952rank,luce2012individual}, is arguably the most well-studied model for ranking problems. This is a parametric model where the strength of the $i$-th item is denoted by $w^*_i>0$, and for each $\set{i,j} \in \calE$ we observe a binary value $y_{ij}$, which takes the value $1$ if $j$ beats $i$ and $0$ otherwise. Moreover, it is assumed that the variables $(y_{ij})_{\set{i,j}\in\calE}$ are independent and distributed as Bernoulli random variables with parameters $\frac{w^*_i}{w^*_i+w^*_j}$. Numerous methods have been proposed to analyze the model theoretically with the goal of estimating the latent strengths. Perhaps the two methods that have received the most attention are the maximum likelihood estimator (MLE) \cite{bradley1952rank,negahban2015rank}, and the Rank Centrality method introduced by Neghaban et al. \cite{negahban2015rank}. Both these methods have been analyzed for general graph topologies (eg., \cite{negahban2015rank}) and for the special case of Erdös-Renyi graphs (eg., \cite{Chen_2019,chen2020partial}). As mentioned earlier, the BTL model (which is nonlinear) can be approximated by the TransSync model \eqref{eq:intro_tran_sync} by using the transformation $z^*_i=\log{w^*_i}$ for each item $i$, along with some additional technical considerations (see Remark \ref{rem:btl_setup} for details).  Other alternative models that have been thoroughly studied for the ranking problem are the Thurstone model \cite{thurstone1927law}, models based on the strong stochastic transitivity (SST) property \cite{Chatterjee2015matrixEst,shah2016stochastically,pananjady2017worst} and noisy sorting \cite{Mao2018Noisy}, although for the latter the objective is to recover a ground truth permutation (which is the ranking) rather than a vector of strengths.  

\paragraph{Dynamic ranking.} 
The dynamic setting has been relatively less studied than its static counterpart, at least from a theoretical perspective. 
\begin{itemize}
\item Recently, \cite{karle2021dynamic, bong2020nonparametric} considered a dynamic BTL model where the parameters of the model evolve in a Lipschitz smooth manner. For this setting, pointwise consistency rates were derived for estimating the strength vectors at any given time $t \in [0,1]$ - this was shown for a nearest neighbor based rank centrality method in \cite{karle2021dynamic}, and for a maximum likelihood based approach in \cite{bong2020nonparametric}. Unlike the local smoothness assumptions made in these works, we consider a global smoothness assumption, akin to the quadratic variation of a function over a grid. 

\item Another study of time-dependent pairwise comparisons was proposed in \cite{li2021recovery}. Their model posits that comparisons depend on various factors captured in the matrices $S_k$ and $Q$, where $S_k$ gathers the time-dependent data and $Q$ is a latent feature matrix. More precisely, the pairwise comparison matrix at time $k$, namely $X(k)$, is given by
\[X(k) = S_kQ^\top - QS_k^\top \quad \forall k\in[T], \]
and the matrix $S_k$ is assumed to evolve smoothly with time under the following model
\[S_k = S_{k-1}+E_k \quad \forall k \in [T]. \]
Here $E_k$ is a matrix formed of i.i.d. centered Gaussian variables.  The goal is to recover $X(T)$, given noisy linear measurements of $X(k)$. Note that in the case where $S_k$ is a vector of strengths and $Q$ is the all ones vector, the comparisons $X_{ij}(k)$ are noisy measurements of strength differences, which is similar to the Dynamic TranSync model. However, the smoothness assumption is local, unlike the global smoothness case studied in our work. \rev{This local assumption generally holds in practice, but fails under some real-life events such as injuries for athletes, that can cause jumps in the strengths' smoothness. The global assumption used in this work handles such possibilities.}

\end{itemize}
Apart from these works, some theoretical results have been provided for a state-space generalization of the BTL model \cite{fahrmeir1994dynamic,glickman1998state} and in a Bayesian framework \cite{glickman1998state,lopez2018often,maystre2019pairwise}.
 Other existing works have been devoted mostly to applications in the context of sport tournaments \cite{cattelan2013dynamic,jabin2015continuous,motegi2012network}. 

\paragraph{Group synchronization.}  The model \eqref{eq:intro_tran_sync} can be extended to the case where each item $i$ is associated to $g_i$ which is an element of a group $(\calG,\star)$, and the data available consists of noisy measurements of $g_i\star g_j^{-1}$, where $\star$ is the group's operation.
For instance  $\calG=SO(2)$ is the group of planar rotations, for which the corresponding problem is known as angular synchronization. Here, one is asked to
estimate $n$ unknown angles $\theta_1,\cdots,\theta_n\in [0,2\pi)$ given noisy measurements of their offsets $\theta_i-\theta_j\text{ mod } 2\pi$. This problem can be formulated as an optimization problem involving the maximization of a quadratic function over a non-convex set and is NP-hard. Singer introduced spectral and semi-definite programming (SDP) relaxations \cite{Singer2011angular} which work well empirically, and also come with statistical guarantees \cite{Bandeira2017}. The group synchronization problem has also been studied from a theoretical perspective for other groups such as $SO(d)$ and $\mathbb{Z}_2$ (e.g., \cite{Ling20, mei17a}).

\paragraph{Signal denoising on graphs and smoothness constraints.}
In terms of the smoothness constraints and the proposed estimators, the literature of signal denoising in graphs is perhaps the closest to our work. In particular, we draw inspiration from \cite{sadhanala2016total}, where the authors consider a single graph $G=([n],\calE)$, a ground truth signal $x^*\in \matR^n$ and study the problem of estimating $x^*$ under the observation model 
\begin{equation} \label{eq:sig_den}
y=x^*+\epsilon,
\end{equation}
where $\epsilon$ is centered random noise. Denoting $L$ to be the \rev{unnormalized} Laplacian matrix of $G$, it is assumed that the quadratic variation of $x^*$ (i.e.,  $x^{*T}Lx^*$) is not large which means that the signal does not change quickly between neighboring vertices. This is equivalent to saying that $x^*$ lies close to the subspace spanned by the eigenvectors corresponding to the small eigenvalues of $L$ (i.e., the 
`low frequency' part of $L$). In \cite{sadhanala2016total}, the authors show for sufficiently smooth $x^*$ and with $G$ assumed to be the grid graph that linear estimators such as \emph{Laplacian smoothing} and \emph{Laplacian eigenmaps} attain the minimax rate for estimating $x^*$ in the $\ell_2$ norm. There are two main differences between this problem setting and ours. Firstly, the observations $y_{ij}(t)$ in our setup (defined in \eqref{eq:transync_intro}) are supported on the edges rather than on the vertices of a graph. Secondly, we impose a smoothness constraint based on time and not on vertex vicinity on $G_t$, which is more adapted to the strength evolution in the dynamic ranking problem. Nevertheless, the estimators we propose (see Section \ref{sec:model}) are essentially adaptations of the Laplacian smoothing and Laplacian eigenmaps estimators to our setup.

\subsection{Contributions} 
In this paper, we make the following main contributions. 
\begin{itemize}
\item We consider a dynamic version of the TranSync model, where the latent strength parameters evolve smoothly with time. The smoothness assumption we work under is global in nature, as opposed to local assumptions previously considered in the literature \cite{karle2021dynamic,bong2020nonparametric}. We propose two estimators for this model, one based on a smoothness-penalized least squares approach and the other based on a projection method, which can be considered as extensions of the Laplacian smoothing and Laplacian eigenmaps estimators \cite{sadhanala2016total}, respectively.

\item We provide theoretical guarantees for the proposed estimators, deriving high probability bounds for the global $\ell_2$ estimation error. In particular, assuming $\calT \subset [0,1]$ is a uniform grid ($\abs{\calT} = T+1$) and the comparison graphs $(G_{t'})_{t'\in\calT}$ are connected, we show that under mild smoothness assumptions (see Theorems \ref{thm:error_ls_evol_graph_2} and \ref{thm:error_proj_evol}), the estimate $\est{z}_t$ returned by either of our estimators satisfy \eqref{eq:mse_consis}. 
The connectivity assumption on each individual graph is due to technical reasons for carrying out the analysis, but as discussed later, we believe these results can be extended to handle the general setting where a fraction of the comparison graphs are disconnected.  
%

\item Finally, we provide numerical experiments that complement our theoretical results. We perform extensive experiments with synthetic data for the dynamic versions of the TranSync and BTL models that demonstrate that our estimators perform well even in very sparse regimes where individual graphs can be possibly disconnected. In addition, we evaluate our methods on two real data sets -- the Netflix Prize data set \cite{jiang2011statistical} and Premier League results from season $2000/2001$ to season $2017/2018$ \cite{ranksync21} -- and show that there is a benefit in using temporal information for estimation.
 
\end{itemize}
\paragraph{Outline of the paper.} The remainder of the paper is organized as follows. We present the setup of the dynamic TranSync model in Section \ref{sec:model} and the proposed penalized least squares and projection estimators in Section \ref{sec:estimators}. We introduce our main theoretical results, Theorems \ref{thm:main_result_smooth} and \ref{thm:main_result_proj}, in an informal manner in Section \ref{sec:main_results}. In Sections \ref{sec:analysis_smooth_pen} and \ref{sec:analysis_proj} we present the analysis of the proposed estimators, including the precise statements of our main results and we give their proofs in Section \ref{sec:proofs}. Section \ref{sec:experiments} contains experiments on synthetic and real data sets. Finally, we summarize our findings in Section \ref{sec:conclusion} and give prospective research directions. 

\section{Problem setup and algorithm} \label{prob_setup}

\paragraph{Notation.} For a matrix $A$, we denote $\norm{A}_2$ to be the spectral norm, i.e., the largest singular value, and  $A^\dagger$ to be its Moose-Penrose pseudo-inverse. If $A \in \matR^{n \times n}$ is symmetric, its eigenvalues are ordered such that $\lambda_n(A) \leq \cdots \leq \lambda_1(A)$. Given two symmetric matrices $A,B$, the Loewner ordering $A \preccurlyeq B$ means that $B-A$ is positive semi-definite. The symbol $\otimes$ is used for denoting the usual Kronecker product of two matrices. For real numbers $a,b$, we denote $a\wedge b := \min(a,b)$ and $a\vee b := \max(a,b)$. For $a,b \geq 0$, we write $a \lesssim b$ if there exists a constant $c > 0$ such that $a \leq c b$. Moreover, we write $a\asymp b$ if $a \lesssim b$ and $b \lesssim a$. Finally, $\ones_n$ denotes the all one's vector of size $n$, and $I_{n}$ denotes the $n \times n$ identity matrix.

\subsection{The Dynamic TranSync model}\label{sec:model}
Let us introduce formally our model for dynamic pairwise comparisons, inspired by the TranSync model \cite{huang2017translation}.
Our data consists of pairwise comparisons on a set of items $[n] = \set{1,2,\dots,n}$ at different times $t \in \calT = \set{\frac{k}{T} | \, k=0,\dots,T}$, where $\calT$ is a uniform grid on the interval $[0,1]$. At each time $t \in \calT$, we denote the observed comparison graph $G_t = ([n],\calE_t)$ where $\calE_t$ is the set of undirected edges. It will be useful to denote $\vv{\calE_t} = \{ (i,j) | \, \{i,j\} \in \calE_t, \, i<j \}$ as the corresponding set of directed edges.  We assume that the set of items $[n]$ is the same throughout, but the set of compared items $\calE_t$ can change with time. 

To model our data, we use the TranSync model at each time $t$ which posits that the outcome of a comparison between two items is solely determined by their strengths. The strengths of the items at time $t$ are represented by the vector $z^*_t = (z^*_{t,1}, \dots, z^*_{t,n})^\top \in \matR^n$. For each $t \in \calT$ and for every pair of items  $\set{i,j} \in \calE_t$, we obtain a noisy measurement of the strength difference $z^*_{t,i}-z^*_{t,j}$ 
\begin{equation}\label{eq:transync_model}
y_{ij}(t) = z^*_{t,i}-z^*_{t,j} + \epsilon_{ij}(t),
\end{equation}
where $\epsilon_{ij}(t)$ are i.i.d. centered subgaussian random variables \red{with $\psi_2$ norm  $\norm{\epsilon_{ij}(t)}^2_{\psi_2} = \sigma^2$ \cite{vershynin2010introduction}}. Let us  denote $x^*(t) \in \matR^{\abs{\calE_t}}$ where
%
%
\begin{equation*}
    x^*_{ij}(t) = z^*_{t,i}-z^*_{t,j}, \quad \set{i,j} \in \calE_t
\end{equation*}
and $x^*$ to be formed by column-wise stacking as 
\begin{equation*}
    x^* = \begin{pmatrix}
x^*(0)  \\
x^*(1) \\
\vdots \\
x^*(T)
\end{pmatrix}.
\end{equation*}
\begin{remark}[BTL model]\label{rem:btl_setup}
    A well-studied model in ranking problems is the BTL model \cite{bradley1952rank}, and it has recently been extended to a dynamic setting \cite{bong2020nonparametric,karle2021dynamic}. It posits that at each time $t$, the probability that an item is preferred to another item only depends on their strengths. If $w^*_{t,i}$ is the strength of item $i$ at time $t$, the quantity $\frac{w^*_{t,i}}{w^*_{t,i}+w^*_{t,j}}$ represents the probability that $i$ beats $j$ at time $t$.  More precisely, considering $z^*_{t,i}=\ln{w^*_{t,i}}$, it states that
$$\logit\left(\prob(i \text{ preferred at }j\text{ at time }t)\right) = z^*_{t,i} - z^*_{t,j},$$
where $\logit(x) = \ln\frac{x}{1-x}$ for $x\in (0,1)$. At each time $t\in\calT$ and for each pair $\set{i,j}\in\calE_t$, the observations are $L$ Bernoulli random variables with parameter equal to the probability that $i$ is preferred to $j$ at time $t$. Then, taking $\tilde y_{ij}(t)$ as the mean of all the observations corresponding to this triplet $(t,i,j)$, the fraction $R_{ij}(t) := \frac{\tilde y_{ij}(t)}{\tilde y_{ji}(t)}$ can be viewed as a noisy measurement of $\frac{ w^*_{t,i}}{ w^*_{t,j}}$. Indeed, it was shown in \cite{hendrickx2020minimax} for the analogous static case (we drop $t$ from the notation) with $R_{ij} := \frac{\tilde y_{ij}}{\tilde y_{ji}}$ that  
\begin{equation}\label{eq:BTL_Taylor}
    \ln{R_{ij}} = \ln{w^*_i}-\ln{w^*_j}+\tilde{\epsilon}(w^*_i,w^*_j,y_{ij})
\end{equation}
%
\red{where the RHS of \eqref{eq:BTL_Taylor}  consists of terms in the Taylor expansion of the $\ln$ function.} 
Note that the noise in the observations (\red{captured by $\tilde{\epsilon}$ in \eqref{eq:BTL_Taylor}}) is no longer zero-mean, contrary to the Dynamic TranSync model, which can introduce some bias in the estimation.
\end{remark}
\begin{remark}[\rev{Outliers model}]
    \rev{Another well-known model in group synchronization is the outliers model \cite{Singer2011angular}. It posits that for any pair of compared items $\set{i,j}$, we observe either the true strength difference or random noise. In our setting, one can introduce the following analogous version of this model,
    \[ y_{ij}(t) = (z^*_{t,i}-z^*_{t,j})X_{ij}(t) + (1-X_{ij}(t))\epsilon_{ij}(t)\]
    where $\epsilon_{ij}(t)\sim \calN(0,1)$ and $X_{ij}(t) \sim \calB(\eta)$, $\eta \in ]0,1[$, denotes the probability of observing the true strength difference, independent of $\epsilon_{ij}(t)$. Note that in this model, observations are not centered around $x^*(t)$, contrary to \eqref{eq:transync_model}. However, we can rewrite $y_{ij}(t)$ as 
    \[ {y}_{ij}(t) = y_{ij}(t)  - \expec[y_{ij}(t)]+\expec[y_{ij}(t)] = \eta x^*_{ij}(t) + \underbrace{(1-X_{ij}(t))\epsilon_{ij}(t)+ (X_{ij}(t) - \eta)x^*_{ij}(t)}_{0-\text{mean, subgaussian}}\]
    which is similar to \eqref{eq:transync_model}, but with an additional bias since $\expec[y_{ij}(t)] \neq x^*_{ij}(t)$. This suggests that the analysis of the outliers model could potentially be done following a similar strategy to ours, although the presence of the bias term would likely make the analysis more cumbersome. Note that $\eta$ is unknown, hence one can not use it directly in the estimation procedure.}
\end{remark}

Finally, we mention that since $t = k/T$ for an integer $0 \leq k \leq T$, we will often interchangeably use $t$ and $k$ for indexing purposes.
%
%

\paragraph{Smooth evolution of weights.}  As discussed in the introduction, we will assume that the weights $z^*_t$ do not change, in an appropriate sense, too quickly with $t$. 
Since each $z^*_t$ is only identifiable up to a constant shift, the smoothness assumption that we impose needs to be invariant to such transformations. The assumption we make is as follows.
\begin{assumption}[Global $\ell_2$ smoothness] \label{assump:smooth}
Let $\complincmat \in \matR^{n \times {n \choose 2}}$ to be the edge incidence matrix of the complete graph $K_n$. We assume that
\begin{equation}\label{smooth_assump}
\sum_{k=0}^{T-1} \norm{\complincmat^\top (z^*_k -  z^*_{k+1})}^2_2 \leq S_T.
\end{equation}
\end{assumption}
The above assumption states that the vector $C^\top z^*_t \in \matR^{n \choose 2}$ does not change too quickly ``on average'', and is analogous to the usual notion of quadratic variation of a function. The regime $S_T = O(T)$ is uninteresting of course, what we are interested in is the situation where $S_T = o(T)$. Also note that \eqref{smooth_assump} is invariant to a constant shift of $z^*_t$, as desired. We will then aim to estimate the `block-centered' version of the vector $z^* \in \matR^{n(T+1)}$, where each block $z^*_t$  is shifted by an additive constant such that 
$$\frac{1}{n}\sum_{i=1}^n z^*_{t,i} = 0 \quad \forall \ t \in \calT.$$ 
In what follows, we will assume w.l.o.g that $z^*$ is block-centered. \red{Finally, denoting $M \in \matR^{(T+1) \times T}$ to be the incidence matrix of the path graph on $T+1$ vertices and $E = M^\top \otimes \complincmat^\top$, we can write \eqref{smooth_assump} as 
$$\sum_{k=0}^{T-1} \norm{\complincmat^\top (z^*_k-z^*_{k+1})}_2^2 = \norm{E z^*}_2^2 = z^{*^\top}E^\top Ez^* \leq S_T.$$
Hence $z^*$ lies close to the null space of $E^\top E$, i.e. $\calN(E^\top E)$, where this closeness is captured by $S_T$.}

\rev{\begin{remark}[Other smoothness assumptions and error guarantees]
As discussed earlier in Section \ref{subsec:relwork}, local smoothness assumptions have been studied in the literature for other related models, leading to recovery guarantees at a given time $t$ (see \cite{karle2021dynamic,bong2020nonparametric} for dynamic BTL model and also \cite{li2021recovery} for a non-transitive model with autoregressive noise). Here, we propose a weaker global smoothness assumption, which fits better some real-life data sets, as it allows for jumps in the smoothness. However, under a weaker global assumption, the error guarantees one can provide will also be of the global type, and relatively weaker compared to the local ones. In this work, we will bound the MSE for estimating $z^*$ (recall \eqref{eq:mse_consis}).
\end{remark}}

%
%
%

\subsection{Smoothness-constrained estimators}
\label{sec:estimators}
In the Dynamic TranSync model \eqref{eq:transync_model}, observations are noisy measurements of strength differences with zero-mean noise. A classical approach for recovering the strength vector $z^*$ is to solve the following linear system
\begin{equation}\label{eq:unconstrained_ls}
    y_{ij}(t) = z_{t,j}-z_{t,i}, \quad \forall t \in \calT, \, \set{i,j} \in \calE_t 
\end{equation}
in the least-squares sense.
%
%
To reflect the temporal smoothness of the data, we need to take into account the smoothness constraints from Assumption \ref{assump:smooth}. We will consider two different approaches to incorporate these constraints in \eqref{eq:unconstrained_ls}.
\paragraph{Smoothness-penalized least squares.} A typical approach for incorporating constraints into a least squares problem involves adding them as a penalty term. Hence, an estimator $\est z$ of the strength vector is given as a solution of the following problem.
\begin{equation}\label{eq:pen_ls_estimator}
    \est z = \argmin{\substack{z_0, \dots, z_T \in \matR^n \\  z_k^\top \ones_n = 0}} \sum_{t \in \calT} \sum_{(i,j) \in \vv{\calE_t}}\left( y_{ij}(t) - (z_{t,i}-z_{t,j}) \right)^2 + \lambda \sum_{k=0}^{T-1} \norm{\complincmat^T(z_k - z_{k+1})}_2^2.
\end{equation}
The penalty term promotes smooth solutions and the estimate $\est{z} \in \matR^{n(T+1)}$ is formed by column stacking $\est{z}_0,\dots,\est{z}_{T} \in \matR^{n}$. If $\lambda = 0$, then each estimate $\est{z}_k$ only uses the information available at this time instant, through $G_k$. Hence, the error $\norm{z^*-\est z}_2^2$ will typically grow linearly with $T$ (large variance). On the other hand, if $\lambda$ is very large, then the estimate $\est z$ will be very `smooth' meaning that $\est z_k$ will be similar for all $k$. Hence, the error $\norm{z^*-\est z}_2^2$ will typically have a large bias. Therefore an intermediate choice of the parameter $\lambda$ is important to achieve the right bias-variance trade-off.

\begin{remark}[Laplacian smoothing]
\red{Note that the above estimator has similarities with the Laplacian smoothing estimator  \cite{sadhanala2016total} for the model in \eqref{eq:sig_den}, where an estimate $\hat x$ of $x^*$ is obtained as 
\begin{equation*}
    \hat x = \argmin{x} \norm{y-x}_2^2 + x^\top L x.
\end{equation*}
%
Indeed, the Dynamic TranSync model can be compared with the setting in \cite{sadhanala2016total} by considering each of the graphs $G_t$ as the vertices of a path graph, for which information is available as the vector of observations $y(t)$ at time $t$. This motivates solving the penalized least-square problem in \eqref{eq:pen_ls_estimator} using the smoothness assumption in \eqref{smooth_assump} as the penalty term. Unlike  Laplacian smoothing, the penalty in our case does not involve a Laplacian matrix.}
\end{remark}

\paragraph{Projection method.} Our second approach consists of the following two-stage estimator. 
\begin{enumerate}
    \item \textit{Step $1$}: For each $t \in \calT$, we compute $\check z_t \in \matR^{n}$ as the (minimum $\ell_2$ norm) least-squares solution of \eqref{eq:unconstrained_ls}, and form $\check z \in \matR^{n(T+1)}$ by column-stacking $\check z_0, \dots, \check z_{T}$.
    
    \item \textit{Step $2$}: Let $\lfreqsp$ be the space generated by the eigenvectors of $E^\top E$ corresponding to  eigenvalues smaller than a threshold $\tau>0$, and let $\lfreqproj$ be the projection matrix on $\lfreqsp$. Then, the estimator for $z^*$ is defined as
\begin{equation}\label{eq:proj_estimator}
 \est z = \lfreqproj \check z.
 \end{equation}

\end{enumerate}
As will be seen later due to the form of $\calN(E^\top E)$, it will hold that each block $\est z_k \in \matR^n$ satisfies $\est {z}_k^\top \ones_n = 0$ for $k=0,\dots,T$. 
%
%
\begin{remark}[Laplacian eigenmaps]
\red{The above estimator is constructed in a similar fashion as the Laplacian eigenmaps estimator \cite{sadhanala2016total} for the model \eqref{eq:sig_den}, which is obtained by projection of the observations $y$ onto the space spanned by the smallest eigenvectors (i.e., corresponding to the smallest eigenvalues) of the Laplacian $L$. 
In our setting, we do not have direct information on the vertices so we replace it by a least-squares solution of \eqref{eq:unconstrained_ls} in Step $1$. Moreover, we construct the projection matrix using the smallest eigenvectors of $E^\top E$ since $z^*$ lies close to $\calN(E^\top E)$.}
\end{remark}

\section{Main results}\label{sec:main_results}
Here we present our main results, namely Theorems \ref{thm:main_result_smooth} and \ref{thm:main_result_proj}, which correspond to theoretical guarantees for the estimation error of the proposed estimators $\hat z$ and $\tilde z$, based on smoothness-penalized least squares and the projection method respectively. We choose to write the results in this section in a stylized form for better readability, highlighting the rates with respect to the time parameter $T$ and hiding the dependency on the rest of the parameters, which will be later made explicit in Sections \ref{sec:analysis_smooth_pen} and \ref{sec:analysis_proj}. Our theoretical results hold under the assumption that each comparison graph $G_t$ is connected, but this is essentially for technical reasons and we believe this requirement can be relaxed, see Remarks \ref{rem:connectedness_assum2} and    \ref{rem:connectedness_assum} for a more detailed discussion. Simulation results in Section \ref{sec:synth_exps} show that the MSE goes to zero for both the estimators (as $T$ increases) even in the very sparse regime where individual $G_t$'s may be disconnected. 
\begin{thmx}[Smoothness-penalized least squares]\label{thm:main_result_smooth}
Let $\hat z \in \matR^{n(T+1)}$ be the estimator defined in \eqref{eq:pen_ls_estimator} where the data is generated by the model \eqref{eq:transync_model} with subgaussian noise parameter $\sigma^2$ and with $z^* \in \matR^{n(T+1)}$ as the ground truth vector of strength parameters. Suppose additionally that $G_t$ is connected for each $t \in \calT$. \rev{Under Assumption \ref{assump:smooth}, if $\lambda =\sigma^{\frac45}  (\frac{T}{S_T})^{2/5}$}, it holds with probability larger than $1-\delta$ 
\begin{equation*}
    \|\hat z- z^*\|_2^2\leq T^{\red{4/5}}S^{\red{1/5}}_T\Psi_{\operatorname{LS}}\big(n,\sigma, \delta\big)+\red{\Psi'_{\operatorname{LS}}\big(n,\sigma, \delta\big).}
\end{equation*}
Here, $\Psi_{\operatorname{LS}}(\cdot)$ and $\red{\Psi'_{\operatorname{LS}}(\cdot)}$ are functions of the parameters of the problem. 
\end{thmx}
A more formal version of this theorem is given by Theorem \ref{thm:error_ls_fixed_graph_2} in Section \ref{sec:analysis_smooth_pen}, where the functions $\Psi_{\operatorname{LS}}(\cdot)$, $\Psi'_{\operatorname{LS}}(\cdot)$ are made explicit, up to universal constants. 
\begin{thmx}[Projection method]\label{thm:main_result_proj}
Let $\est z \in \matR^{n(T+1)}$ be the estimator defined in \eqref{eq:proj_estimator}. Assume that each comparison graph  in the sequence $(G_k)^T_{k=0}$ is connected. If $\tau = \rev{\sigma^{-\frac43}}(\frac{S_T}{T})^{2/3}$, then it holds with probability larger than $1-\delta$ that
\begin{equation*}
    \|\est z- z^*\|_2^2\leq T^{2/3}S_T^{1/3}\Psi_{\operatorname{Proj}}\big(n,\sigma, \delta\big) +\red{\Psi'_{\operatorname{Proj}}\big(n,\sigma, \delta\big)}
\end{equation*}
where $\Psi_{\operatorname{Proj}}(\cdot)$ and $\red{\Psi'_{\operatorname{Proj}}(\cdot)}$ are functions of the parameters of the problem. 
\end{thmx}
For a more formal statement of the previous theorem (with the explicit $\Psi_{\operatorname{Proj}}, \Psi'_{\operatorname{Proj}}$), we direct the reader to Theorem \ref{thm:error_proj_evol} in Section \ref{sec:analysis_proj}. Observe that in the case $S_T=O(T)$, the error in both theorems is of order $O(T)$, which is to be expected, since in this case the smoothness constraint is always satisfied and the statement becomes vacuous. 
When $S_T = o(T)$, we see for both estimators that the mean squared error 
\begin{equation*}
    \frac{1}{T+1} \norm{\est{z} - z^*}_2^2 = \frac{1}{T+1} \sum_{k=0}^{T+1} \norm{\est{z}_k - z^*_k}_2^2 = o(1) \quad \text{as} \quad T \rightarrow \infty.
\end{equation*}
\red{However the bound in Theorem \ref{thm:main_result_proj} is better than that in Theorem \ref{thm:main_result_smooth} in terms of dependence on $T$, when $S_T = o(T)$. This is due to technical difficulties arising during the control of the bias term in the proof of Theorem \ref{thm:main_result_smooth}, since $(G_k)_{k=0}^T$ is allowed to be any sequence of connected graphs. In spite of this, we are able to obtain the same error-rate as Theorem \ref{thm:main_result_proj} under additional assumptions. More specifically, in Theorem \ref{thm:error_ls_fixed_graph_2} we assume that the comparison graphs are non-evolving (see Assumption \ref{assump:fixed_graph}) and obtain a rate that matches that of Theorem \ref{thm:main_result_proj}. In Theorem \ref{thm:error_ls_evol_graph_3}, we prove that if the comparison graphs are connected and possibly evolving, and an additional technical condition holds, then we again recover the rate of the projection estimator.} 
%
\begin{remark}[\rev{Error rates -- Lipschitz smoothness}]
 When $S_T=O(1/T)$, then $\norm{\est{z} - z^*}_2^2 = O(T^{1/3})$ for Theorem \ref{thm:main_result_proj}, which matches the optimal rate for estimating a Lipschitz function on $[0,1]$ over a uniform grid, w.r.t the squared (empirical) $L_2$-norm  \cite[Thm.1.3.1]{nemirovski2000topics}. Indeed, let $f_i:[0,1] \rightarrow
 \matR$ be Lipschitz functions for $i=1,\dots,n$ and define $f(t) = (f_1(t),\dots,f_n(t))^\top$ for each $t \in [0,1]$. Then consider the case where $z^*$ is defined as 
 \begin{equation*}
     z^*_{t,i} := f_i(t) - \frac{\ones^\top f(t)}{n}; \quad i \in [n], t \in \calT.
 \end{equation*}
Clearly, $z^*$ is block-wise centered and satisfies Assumption \ref{assump:smooth} with $S_T = O(1/T)$. 
\end{remark}

\section{Smoothness-penalized least squares estimator analysis}\label{sec:analysis_smooth_pen}
In this section we obtain theoretical guarantees for the error of the estimator $\hat{z}$ given in \eqref{eq:pen_ls_estimator}. The results in this section will build towards proving Theorem \ref{thm:main_result_smooth}, starting from the special case when all the graphs are equal in Section \ref{sec:warm_up_ls}. Before proceeding, we  introduce some notation that will appear in the analysis.  
\paragraph{Notation.} Let us denote by $Q_k$ the incidence matrix of the graph $G_k$ and $Q$ will denote the $n(T+1)\times \sum_{k=0}^T |\calE_k|$ block diagonal matrix where the blocks on the diagonal are the matrices $Q_k$. Similarly, we define the Laplacian at time $t=k/T$ by $L_k:=Q_kQ_k^\top$ and $L$ will be defined as the stacked Laplacian, which is the $n(T+1)\times n(T+1)$ block diagonal matrix with blocks $L_k$ on the main diagonal. We define, for $\lambda > 0$, the regularized Laplacian matrix $L(\lambda):=L+\lambda E^\top E$. Let us also define notation for the eigenpairs of matrices $CC^\top, MM^\top$ and $L_k$. 
\begin{itemize}
    \item $(\lambda_j,v_j)_{j=1}^n$ denotes the eigenpairs of $CC^\top$. Observe that $v_n = \frac{1}{n}\ones_n$ and $(v_j)_{j=1}^{n-1}$ can be any orthonormal basis of $\spn(\ones_n)^\perp$. In addition, we have $\lambda_j=n-1$ for $1\leq j\leq n-1$ and $\lambda_n=0$.
    
    \item $(\mu_k,u_k)_{k=0}^T$ denotes the eigenpairs of $MM^\top$ (path graph on $T+1$ vertices), with $\mu_0 \geq \mu_1\geq \mu_2\geq \cdots \geq \mu_{T-1} > \mu_T=0$. Note that $u_{T} = \spn(\ones_{T+1})$.
    
    \item $(\alpha_{k,j},a_{k,j})_{j=1}^n$ denotes the eigenpairs of $L_k$, with $\alpha_{k,1}\geq \alpha_{k,2}\geq \cdots\geq\alpha_{k,n}=0$ for all $k=0,\dots,T$.
\end{itemize}
Now, the estimator $\hat z$ defined in \eqref{eq:pen_ls_estimator} can be equivalently defined as a solution of 
\begin{equation} \label{eq:ls_matrix}
    \min_{\substack{z \in \matR^{n(T+1)}, \\ z_k^\top \ones_n = 0}} \norm{Q^\top z - y}_2^2 + \lambda \norm{E z}_2^2.
\end{equation}
For the unconstrained problem, the solutions of \eqref{eq:ls_matrix} satisfy
\begin{equation}\label{eq:ls_est_mat_eq}
L(\lambda)z = Qy.
\end{equation}
The following lemma provides conditions under which the null space $\calN(L(\lambda)) = \spn{\set{e_k \otimes \ones_n}_{k=0}^{T}}$ where $e_0,\dots,e_{T} \in \matR^{T+1}$ is a canonical basis of $\matR^{T+1}$. 
\begin{lemma} \label{lem:rank_Laplacian_reg}
If the union graph $G_U := ([n], \cup_{t \in \calT}\calE_t)$ is connected, it follows for any $\lambda > 0$ that $\calN\big(L(\lambda)\big)=\spn{\set{e_k \otimes \ones_n}_{k=0}^{T}}$. Here, $e_0,\dots,e_{T} \in \matR^{T+1}$ is a canonical basis of $\matR^{T+1}$.
\end{lemma}
The proof is outlined in Appendix \ref{appsubsec:rank_Laplacian_reg}. It follows that if the union graph $G_U$ is connected, then the estimator $\est{z}$ is given by the following solution of \eqref{eq:ls_est_mat_eq}
\begin{equation}\label{eq:ls_est_mat}
    \est{z} = L^\dagger(\lambda) Qy,
\end{equation}
which is uniquely defined and is orthogonal to $\spn{\set{e_k \otimes \ones_n}_{k=0}^{T}}$ (in other words, $\hat z$ is block-wise centered).
Furthermore, the ground truth $z^*$ satisfies by definition $Lz^*=Qx^*$, which implies that $L(\lambda)z^*=Qx^*+\lambda E^\top Ez^*$. Since $z^*$ is block-wise centered, i.e., $z^* \perp \calN(L(\lambda))$ due to Lemma \ref{lem:rank_Laplacian_reg}, hence it satisfies
\begin{equation}\label{eq:ls_ground_mat}
    z^* = L^\dagger(\lambda)\big(Qx^*+\lambda E^\top Ez^*\big).
\end{equation}
By \eqref{eq:ls_est_mat}, \eqref{eq:ls_ground_mat} and the triangle inequality, we arrive at the following bound
\begin{equation}\label{eq:error_ls}
    \|\hat{z}-z^*\|_2^2\lesssim \|\Ldag Q(y-x^*)\|_2^2+\lambda^2\|\Ldag E^\top Ez^*\|_2^2.
\end{equation}
The first term in the RHS of \eqref{eq:error_ls} is the variance term due to noise that will be controlled by a large enough value of $\lambda$. The second term is the bias which depends on the smoothness of $z^*$, and will be controlled by choosing $\lambda$ to be suitably small. The optimal choice of $\lambda$ will then achieve the right bias-variance trade-off.
\subsection{Warm-up: the non-evolving case}\label{sec:warm_up_ls}
We will first analyze the case where the comparison graph is the same across all times, \rev{as} this case will serve as a foundation for our analysis of the time-evolving case. \rev{We will refer to this case as non-evolving, which is formally defined by the following assumption (notice that this is different from what we referred to previously as the static case, where $T=1$).}
\begin{assumption}[Fixed and connected comparison graph] \label{assump:fixed_graph}
Let $G_0$ be any connected graph on $n$ vertices, we assume that $G_k=G_0$, for all $k =  0,\dots, T$.
\end{assumption}
Under Assumption \ref{assump:fixed_graph}, we have $L_k=L_0$ and $Q_k=Q_0$, which implies that the matrices $Q, L$ and $L(\lambda)$ can be written as
\begin{align*}
Q&=I_{T+1} \otimes Q_0,\\
L&=I_{T+1}\otimes L_0,\\
 L(\lambda) &= I_{T+1}\otimes L_0+\lambda E^\top E.
\end{align*}
\paragraph{Spectral decomposition of $L(\lambda)$.} In order to obtain an explicit spectral decomposition for $ L(\lambda) $,  we choose $v_j := a_{0,j}$ (given that $a_{k,j}=a_{0,j}$ we will simply write $a_j$ in the sequel) for all $j=1,\dots, n-1$ as the eigenvector basis for $\Lcom:=CC^T$. With this choice, and recalling that $E^\top E=MM^\top\otimes \Lcom$, the eigenpairs of $E^\top E$ associated with nonzero eigenvalues are 
\[\big((n-1)\mu_k, u_k\otimes a_j\big)_{j=1,\cdots,n-1,k=0,\cdots,T-1}.\] Therefore each eigenvalue of the form $(n-1)\mu_k$ has multiplicity $n-1$. Observe that we have the freedom to choose any orthonormal basis as the eigenbasis for $I_{T+1}$,  we opt for the basis given by $\{u_k\}^T_{k=0}$, i.e., the eigenvectors of $MM^T$. Hence, the eigenpairs of $I_{T+1}\otimes L_0$ for nonzero eigenvalues are given by \[(\alpha_j,u_k\otimes a_j)_{j=1,\cdots,n-1,k=0,\cdots,T}.\] Thus, the following decompositions hold.
%
\begin{align}\label{eq:decomp_spec}
     E^\top E&= \sum_{k=0}^{T-1}\sum_{j=1}^{n-1} (n-1)\mu_k(u_ku_k^\top \otimes a_ja_j^\top),\\
     \label{eq:decomp_L}
     L&=\sum^T_{k=0}\sum_{j=1}^{n-1} \alpha_j(u_ku_k^\top \otimes a_ja_j^\top),\\ \label{eq:decomp_L_lambda}
     L(\lambda)&= \sum_{k=0}^{T-1}\sum_{j=1}^{n-1} (\alpha_j + (n-1)\lambda\mu_k)(u_ku_k^\top \otimes a_ja_j^\top) + \sum_{j=1}^{n-1} \alpha_j(u_Tu_T^\top \otimes a_ja_j^\top).
\end{align}
From \eqref{eq:decomp_L_lambda} we see directly that $L(\lambda)$ has rank $(n-1)(T+1)$ and that its nullspace is given by $\spn\{u_k\otimes a_n\}^T_{k=0} = \spn\set{e_k \otimes \ones_n}_{k=0}^{T}$. Given the above notations and setup, we can now present the following bound on the estimation error.
\begin{proposition}\label{thm:error_ls_fixed_graph} Take $\delta\in (0,e^{-1})$, then under Assumptions \ref{assump:smooth} and \ref{assump:fixed_graph} it holds with probability larger than $1-\delta$ 
\begin{equation}\label{eq:error_ls_tail}
\|\hat z-z^*\|_2^2 \leq \Big(\frac{1}{\alpha^2_{n-1}}\vee 1\Big)\lambda S_T+\sigma^2 \alpha_1 (1+4\log(1/\delta)) \Big(\sum^{T-1}_{k=0}\sum^{n-1}_{j=1}\frac{1}{(\alpha_j+\lambda (n-1)\mu_k)^2}+\sum^{n-1}_{j=1}\frac{1}{{\alpha_j^2}}\Big).
\end{equation}
\end{proposition}
The proof is detailed in Section \ref{sec:prop_fix_graph_pen_error}.  

\paragraph{Choice of $\lambda$.} The right hand side of the estimation error bound \eqref{eq:error_ls} can be regarded as the sum of a bias and a variance term, representing an instance of the bias-variance trade-off phenomenon. Proposition \ref{thm:error_ls_fixed_graph} gives an error bound where the dependence on $\lambda$ (and the other parameters of the problem) is explicit. The following lemma helps us  further simplify the dependence on $\lambda$ for the variance term (second term in the RHS of \eqref{eq:error_ls_tail}). 
\begin{lemma}\label{lem:var_term}
We have 
\begin{equation*} 
\sum^{T-1}_{k=0}\sum^{n-1}_{j=1}\frac1{(\alpha_j+\lambda (n-1)\mu_k)^2}\lesssim \rev{\frac{T\sqrt{n-1}}{\alpha_{n-1}^{3/2}\sqrt{\lambda}}}
\end{equation*}
\end{lemma}
The proof is outlined in Section \ref{sec:proof_lem_var_term}. 
Combining the results of Proposition \ref{thm:error_ls_fixed_graph} and Lemma \ref{lem:var_term}, we obtain
\begin{equation*}
\|\hat z-z^*\|_2^2 \lesssim O\big(\lambda S_T+\rev{\sigma^2}\frac{T}{\sqrt{\lambda}}\big) + \sigma^2 \alpha_1 \log(1/\delta) \sum^{n-1}_{j=1}\frac{1}{{\alpha_j^2}},
\end{equation*}
where we recall that the asymptotic notation $ O(\cdot)$ hides the dependence on constants and all parameters except for $T$ \rev{and $\sigma$}. It is easy to see that the optimal choice for $\lambda$ (in terms of $T, S_T$ \rev{and $\sigma$}) is then given by \begin{equation}\label{eq:lambda_choice}
    \rev{\lambda=\argmin{\lambda'}{\ \lambda' S_T+\sigma^2\frac{T}{\sqrt{\lambda'}}},}
\end{equation} which corresponds to $\lambda=\rev{\sigma^{\frac43}}(\frac{T}{S_T})^{2/3}$. \rev{The choice of $\lambda$ given by \eqref{eq:lambda_choice} is well defined (there is a unique minimizer) in the case where at least one of the terms $S_T$ and $\sigma$ is non-zero. On the other hand, for the case $\sigma=S_T=0$, the estimation error is zero for any choice of $\lambda$, as is clear from Proposition \ref{thm:error_ls_fixed_graph}. We will only consider the case when at least one of $\sigma,S_T$ is non-zero in the sequel.}
Plugging this in \eqref{eq:error_ls_tail} and using Lemma \ref{lem:var_term} we arrive at the following bound on the estimation error. 
\begin{theorem}\label{thm:error_ls_fixed_graph_2}
Let $\delta \in (0,e^{-1})$. Under Assumptions \ref{assump:smooth} and \ref{assump:fixed_graph}, choosing $\lambda=\rev{\sigma^{\frac43}}(\frac{T}{S_T})^{2/3}$, it holds with probability larger than $1-\delta$ that
\begin{equation*}
\|\hat z-z^*\|_2^2\lesssim \rev{\sigma^{\frac43}}T^{\frac23}S_T^{\frac13}\Big(\frac{1}{\alpha_{n-1}^2}\vee 1+\frac{\alpha_1 \sqrt{n-1}} {\alpha^{3/2}_{n-1}}\log{(1/\delta)}\Big)+ \sigma^2\alpha_1 \log{(1/\delta)}\sum^{n-1}_{j=1}\frac{1}{{\alpha_j^2}}.
\end{equation*}
\end{theorem}
The following corollary states the order of the estimation error under different smoothness regimes. 
\begin{corollary}\label{cor:error_ls_fixed_graph}
Assume that $S_T = O(T^\gamma)$ for some $\gamma<1$. If $\lambda=\rev{\sigma^{\frac43}}(\frac{T}{S_T})^{2/3}$, then it holds with high probability that
$\|\hat z-z^*\|_2^2 = O(\rev{\sigma^{\frac43}}T^{\frac{2+\gamma}3}\red{\vee \rev{\sigma^2}}).$
In particular, 
\begin{equation*}
\|\hat z-z^*\|_2^2=
\begin{cases}
O(\rev{\sigma^{\frac43}}T^{\frac23})\text{ if }S_T=O(1) \text{ and } \lambda=\rev{\sigma^{\frac43}}T^{2/3},\\
O(\rev{\sigma^{\frac43}}T^{\frac13})\text{ if }S_T=O(\frac1T) \text{ and } \lambda=\rev{\sigma^{\frac43}}T^{4/3}.\\
\end{cases}
\end{equation*}
\end{corollary}
%

\subsection{Any sequence of connected graphs}\label{sec:evoling_Gt_ls}
We now treat the case where the comparison graphs $G_k$ can differ from one another across time, but we assume that $G_k$ is connected for all $k \in \set{0,\dots,T}$. The main idea is to use the results from Section \ref{sec:warm_up_ls} and the fact that the Laplacian of every connected graph can be bounded (in the Loewner order) by the Laplacian of the complete graph, up to a scalar factor. 

\paragraph{A Loewner bound.} If graph $G_k$ is connected for all $k \in \set{0,\dots,T}$, then there exists a sequence $\beta_0,\cdots,\beta_T$ of strictly positive real numbers such that 
\begin{equation}\label{eq:lowner_bound_0}
  L_k \succcurlyeq  \beta_k \Lcom,
\end{equation}
where $\Lcom = CC^T$ is the Laplacian matrix of $K_n$, the complete graph on $n$ vertices, and $L_k$ is the Laplacian of the graph $G_k$. Notice that when $G_k$ is connected, we can always choose $\beta_k(n-1)$ to be the Fiedler eigenvalue (the second-smallest eigenvalue\footnote{It is known for a fact that between all the connected graphs on $n$ vertices, the one that has the smallest Fiedler value is the path graph, which provides a lower bound of order $n^{-2}$ on $\beta_k(n-1)$, for all $k \in \{0,\cdots,T\}$.}) of $L_k$. Denoting $D_\beta = \diag(\beta_0,\dots,\beta_T)$, we  have the following semi-definite bound for the stacked Laplacian $L$ 
\begin{equation}\label{eq:lowner_bound}
L \succcurlyeq D_\beta\otimes \Lcom \succcurlyeq \min_{0\leq k\leq T}{\beta_k}\cdot (I_{T+1}\otimes \Lcom) := \min_{0\leq k\leq T}{\beta_k} \tilde{L},
\end{equation}
which implies that 
\begin{equation}\label{eq:lowner_bound_2}
L(\lambda) \succcurlyeq (D_\beta+\lambda MM^T)\otimes \Lcom \succcurlyeq \big(\min_{0\leq k\leq T}{\beta_k}\cdot I_{T+1}+\lambda MM^T\big)\otimes \Lcom.
\end{equation}
Let $\beta_k=\lamin(L_k)/(n-1)$, where $\lamin(L_k)$ denotes the smallest nonzero eigenvalue of $L_k$. Then, $\min_{0\leq k\leq T}{\beta_k} = \frac{\lamin(L)}{n-1}$, where $\lamin(L)$ is the smallest nonzero eigenvalue of $L$, and \eqref{eq:lowner_bound_2} implies
\begin{equation} \label{eq:loewner_bd_3}
L(\lambda) \succcurlyeq \Big(\frac{1}{n-1}\lamin(L)\cdot I_{T+1}+\lambda MM^T\Big)\otimes \Lcom := \tilde{L}(\lambda).
\end{equation}
Using these observations, we arrive at the following error bound which is a generalization of Proposition \ref{thm:error_ls_fixed_graph} to the general case of evolving graphs. Its proof can be found in Section \ref{sec:proof_prop2}.
\begin{proposition}\label{thm:error_ls_evol_graph}
Assume that $G_k$ is connected for all $k = 0,\dots,T$. Then under Assumption \ref{assump:smooth}, the estimator $\hat z$ given by \eqref{eq:pen_ls_estimator} satisfies with probability larger than $1-\delta$
\begin{equation*}
    \|\est z-z^*\|_2^2 \lesssim \red{\frac{1}{\lamin^2(L)} \lambda^2 n S_T} + \sigma^2\|L\|_2(1+4\log(1/\delta))\Big(\sum^{T-1}_{k=0}\frac{n-1}{(\lamin(L)+(n-1)\lambda\mu_k)^2}+\frac{n-1}{\lamin^2(L)}\Big).
\end{equation*}
\end{proposition}
\red{The main difference between the previous result and Proposition \ref{thm:error_ls_fixed_graph} is that here we obtain a bound of order $O(\lambda^2S_T)$ for the bias term, while Proposition \ref{thm:error_ls_fixed_graph} gives a bound of order $O(\lambda S_T)$ for the same term. This will have an impact on the final error rate given in Theorem \ref{thm:error_ls_evol_graph_2} below.}

We choose $\lambda$ by proceeding as in Section \ref{sec:warm_up_ls}. \rev{Indeed, we obtain the bound
\begin{equation} \label{eq:varterm_bd_evol_case}
    \sum^{T-1}_{k=0}\frac{n-1}{(\lamin(L)+(n-1)\lambda\mu_k)^2} \lesssim \frac{T\sqrt{n-1}}{\lamin(L)^{3/2}\sqrt{\lambda}}
\end{equation}
in the same manner as in \eqref{eq:proof_intermed}.} Plugging it in Proposition \ref{thm:error_ls_evol_graph}, we \rev{deduce $\lambda = \sigma^{\frac45} \big(\frac{T}{S_T}\big)^{\frac{2}{\red{5}}}$ to be the optimal choice.} This then leads to the following theorem which corresponds to a formal version of Theorem \ref{thm:main_result_smooth}. 
%
%
\begin{theorem}\label{thm:error_ls_evol_graph_2}
Under the assumptions of Proposition \ref{thm:error_ls_evol_graph} and choosing $\lambda = \rev{\sigma^{\frac45}} \big(\frac{T}{S_T}\big)^{\frac{2}{\red{5}}}$, it holds with probability larger than $1-\delta$
\begin{equation*}
    \|\est{z}-z^*\|_2^2\lesssim \rev{\sigma^{\frac85}} T^{\red{\frac45}}S^{\red{\frac15}}_T\left(\frac{n}{\lamin^2(L)} + \frac{\sqrt{n-1} \|L\|_2 \log{(1/\delta)}}{\lamin^{3/2}(L)}\right) + \frac{(n-1) \|L\|_2\sigma^2 \log{(1/\delta)}}{\lamin^{2}(L)}.
\end{equation*}
Consequently, if $S_T = O(T^\gamma)$ for some $\gamma < 1$, then $\|\est z-z^*\|_2^2 = O(\rev{\sigma^{\frac85}} T^{\frac{\red{4}+\gamma}{\red{5}}}\red{\vee \rev{\sigma^2}})$.
In particular, 
\begin{equation*}
\|\hat z-z^*\|_2^2=
\begin{cases}
O(\rev{\sigma^{\frac85}}T^{\red{\frac45}})\text{ if }S_T=O(1) \text{ and } \lambda=\rev{\sigma^{\frac45}}T^{2/\red{5}},\\
O(\rev{\sigma^{\frac85}}T^{\red{\frac35}})\text{ if }S_T=O(\frac1T) \text{ and } \lambda=\rev{\sigma^{\frac45}}T^{4/\red{5}}.\\
\end{cases}
\end{equation*}
\end{theorem}
This theorem reveals the dependence on the parameters of the problem, up to absolute constants. In terms of the notation introduced in the statement of Theorem \ref{thm:main_result_smooth}, we have 
\begin{align*}
\Psi_{\operatorname{LS}}\big(n,\sigma, \delta\big)&=\rev{\sigma^{\frac85}} \Big(\frac{n}{\lamin^2(L)} + \frac{\|L\|_2  \sqrt{n-1} \log{(1/\delta)}}{\lamin^{3/2}(L)}\Big ),\\
\red{\Psi'_{\operatorname{LS}}\big(n,\sigma, \delta\big)}&=\red{
\frac{(n-1) \|L\|_2\sigma^2 \log{(1/\delta)}}{\lamin^{2}(L)}
}.
\end{align*}
\red{The error rates presented in Theorem \ref{thm:error_ls_evol_graph_2} are not optimal and they do not match those obtained in Theorem \ref{thm:error_ls_fixed_graph} for the case of non-evolving graphs. We believe that the correct bound for the square error in this case should be, as in Theorem \ref{thm:error_ls_fixed_graph}, of order $O(T^{2/3}S^{1/3}_T\vee 1)$. In particular, in the case $S_T=O(1)$ we obtain a rate $O(T^{4/5})$ which should be $O(T^{2/3})$. As will be seen in the proof of Proposition \ref{thm:error_ls_evol_graph}, the main technical difficulty to obtain what we believe to be the correct bound lies in the fact that the Loewner ordering is not preserved after taking matrix squares, i.e., \eqref{eq:loewner_bd_3} does not automatically imply  $L^2(\lambda)\succcurlyeq \tilde{L}^2(\lambda)$. On the other hand, under the following assumption on the stacked Laplacian of the comparison graphs, we will obtain the same rate as in the non-evolving case. 
}
\red{
\begin{assumption}\label{assump:technical_lower_bound}
There exists a constant $\kappa>1$ such that for all $\lambda>0$ it holds 
\begin{equation}\label{eq:assump_technical}
\frac1{\kappa}\big(L^2+\lambda^2(E^TE)^2\big)+\lambda(E^TEL+LE^TE)\succcurlyeq 0.
\end{equation}
\end{assumption}
}
\red{Altough technical in nature, we argue that Assumption \ref{assump:technical_lower_bound} is reasonable and rather mild. Indeed, notice that given any pair of symmetric matrices (not necessarily p.s.d) $A,B$ of the same size we have that $\frac1{\kappa}(A^2+B^2)+AB+BA \succcurlyeq 0$ for any $\kappa\leq 1$. Notice also that in the non-evolving case we have that $E^TEL+LE^TE \succcurlyeq 0$  (as can be seen from \eqref{eq:decomp_spec} and \eqref{eq:decomp_L}) which implies in that case the stronger condition that \eqref{eq:assump_technical} is satisfied for every $\lambda>0$ and every $ \kappa>1$ (hence in particular Assumption \ref{assump:technical_lower_bound} holds). It is possible that Assumption \ref{assump:technical_lower_bound} holds for an absolute constant $\kappa>1$ for any stacked Laplacian $L$ of connected graphs $G_k$, but we are unable to prove it. The following theorem shows that under Assumption \ref{assump:technical_lower_bound} we obtain error rates that match those of Theorem \ref{thm:error_ls_fixed_graph_2}.}
\red{
\begin{theorem}\label{thm:error_ls_evol_graph_3}
Under the hypotheses of Proposition \ref{thm:error_ls_evol_graph}, suppose that Assumption \ref{assump:technical_lower_bound} holds. Then choosing $\lambda = \rev{\sigma^{\frac43}}\big(\frac{T}{S_T}\big)^{\frac{2}3}$, it holds with probability greater than $1-\delta$ that
\begin{equation*}
    \|\est{z}-z^*\|_2^2\lesssim\frac{\rev{\sigma^{\frac43}}\okap}{\okap-1} T^{\frac23}S^{\frac13}_T\left(\left(\frac{1}{\lamin^2(L)}\vee1 \right) + \frac{\sqrt{n-1} \|L\|_2\sigma^2 \log{(1/\delta)}}{\lamin^{3/2}(L)}\right) + \frac{(n-1) \|L\|_2\sigma^2 \log{(1/\delta)}}{\lamin^{2}(L)}, 
\end{equation*}
\[\text{ where } \quad \okap:=\max\Big\{\kappa>1 : \quad \frac1{\kappa}\big(L^2+\lambda^2(E^TE)^2\big)+\lambda(E^TEL+LE^TE)\succcurlyeq 0\Big\}.\]
Consequently, if $S_T = O(T^\gamma)$ for some $\gamma < 1$, then $\|\est z-z^*\|_2^2 = O(\rev{\sigma^{\frac43}} T^{\frac{2+\gamma}{3}}\vee \rev{\sigma^2})$.
In particular, 
\begin{equation*}
\|\hat z-z^*\|_2^2=
\begin{cases}
O(\rev{\sigma^{\frac43}} T^{\frac23})\text{ if }S_T=O(1) \text{ and } \lambda=\rev{\sigma^{\frac43}} T^{2/3},\\
O(\rev{\sigma^{\frac43}} T^{\frac13})\text{ if }S_T=O(\frac1T) \text{ and } \lambda=\rev{\sigma^{\frac43}} T^{4/3}.\\
\end{cases}
\end{equation*}
\end{theorem}}

\red{The proof of Theorem \ref{thm:error_ls_evol_graph_3} mimics the argument used to prove Theorem \ref{thm:error_ls_evol_graph_2}. It follows from the following bound, which is analogous to the bound in Proposition \ref{thm:error_ls_evol_graph}
\begin{align}
    \|\est z-z^*\|_2^2 &\leq \frac{\okap}{\okap-1}\Big(\frac{1}{\lamin^2(L)}\vee 1\Big) \lambda S_T \nonumber \\
    &+ \sigma^2\|L\|_2(1+4\log(1/\delta))\Big(\sum^{T-1}_{k=0}\frac{n-1}{(\lamin(L)+(n-1)\lambda\mu_k)^2}+\frac{n-1}{\lamin^2(L)}\Big). \label{eq:error_bound_analog_a3}
\end{align}
\rev{Then using the bound in \eqref{eq:varterm_bd_evol_case}, the optimal choice for  $\lambda$ follows by an elementary calculation}. The bound \eqref{eq:error_bound_analog_a3} is proven by using the second statement of Lemma \ref{lem:bound_bias_non_evol} (which controls the bias term), and Lemma \ref{lem:bound_var_non_evol} is used to bound the variance term (which is unchanged with respect to Theorem \ref{thm:error_ls_evol_graph_2}). These lemmas can be found in Section \ref{sec:proof_prop2}.
}
%
%

\begin{remark}[Connectedness assumption]\label{rem:connectedness_assum2}
We believe that the analysis can likely be extended to the setup where some of the graphs are disconnected (with the union graph $G_U$ being connected). The main technical difficulty in this case comes from the fact that the Loewner bound \eqref{eq:lowner_bound_0} now does not hold for a strictly positive $\beta_k$ for all $k$. If some of the $\beta_k$'s in \eqref{eq:lowner_bound_0} are allowed to be $0$, then $\min_{0\leq k \leq T}\beta_k=0$ which renders the bound \eqref{eq:lowner_bound_2} to be not useful. In the proof of Lemma \ref{lem:bound_var_non_evol}, we need to have a good lower bound on the eigenvalues of $L(\lambda)$ in order to bound $\Tr\big({L}^{\dagger 2}(\lambda)\big)$, but this appears to be very challenging. Moreover, while the first part of Lemma \ref{lem:bound_bias_non_evol} will remain unchanged, it is  difficult to see how the second part of Lemma \ref{lem:bound_bias_non_evol} can be adapted for this setup.
\end{remark}

\section{Projection method analysis}\label{sec:analysis_proj}
Let us now obtain bounds for the estimation error of the projection method. To prove Theorem \ref{thm:main_result_proj}, we will use an analogous scheme to Section \rev{\ref{sec:analysis_smooth_pen}}, i.e., we build upon the \rev{non-evolving} case, \rev{by considering that} all the comparison graphs are the same. Before proceeding, it will be useful to introduce some notation specific to the projection approach that will be used in the analysis. %
\paragraph{Notation.} 
Let us define the following quantities.
\begin{itemize}
\item $\calL_\tau=\{(k,j)\in\{0,\cdots, T\}\times [n]\text{ s.t }\lambda_j\mu_k\rev{\leq}\tau\}$, where $\tau\rev{\in[0,\infty]}$, corresponds to the indices for the low-frequency part of the spectrum \rev{of $E^TE$. Note that when $\tau=0$, it corresponds to the indices for the nullspace of $E^TE$. On the other hand, when $\tau=\infty$, we have $\calL_\infty~=~\{0,\cdots, T\}\times [n]$}. 
\item $\calH_\tau=\{0,\cdots, T\}\times [n]\setminus \calL_\tau$ corresponds to the indices for the high-frequency part of the spectrum.

\item $\lfreqsp=\spn\{u_k\otimes a_j\}_{(k,j)\in \calL_\tau}$ is the linear space spanned by the low frequency eigenvectors. $\lfreqproj$ is the projection matrix for $\lfreqsp$.  

\item $\hfreqproj$ is the projection matrix for the orthogonal complement of $\calV_\tau$, i.e., $\hfreqsp$. 
\end{itemize}
The previous definitions help formalize the description of $\est z$ in \eqref{eq:proj_estimator} and we can write, in matrix notation, 
\begin{equation*}
    \check{z}=L^\dagger Qy \quad \text{ and } \quad  \est{z}=\lfreqproj \check{z}.
\end{equation*}
%
Notice that $\calN(E^\top E) \subset \calV_\tau$  and $\check z \perp e_k \otimes \ones_n$ for each $k$. This then implies that 
\begin{equation*}
    \est{z}=\lfreqproj\check{z} \perp (e_k \otimes \ones_n), \quad k=0,\dots,T,
\end{equation*}
which means that $\est z$ is block-wise centered.  
Now the latent strength vector $z^*$ is given as the solution of the linear system $L z = Q x^*$ with $z^*$ assumed to be block-centered. If each $G_k$ is connected, then $z^*$ is uniquely given by 
\begin{equation*}
 z^* = L^{\dagger} Q x^*.   
\end{equation*}
%
Therefore, since
\begin{equation*}
 z^* =  \lfreqproj z^* +  \hfreqproj z^* =  \lfreqproj L^{\dagger} Q x^* + \hfreqproj z^*,
\end{equation*}
we obtain the following expression for the estimation error 
\begin{equation}\label{eq:error_proj}
\|\est z-z^*\|_2^2=\|\lfreqproj L^\dagger Q(y-x^*) \|_2^2+\|\hfreqproj z^*\|_2^2.
\end{equation}
The first term in the RHS of \eqref{eq:error_proj} is the variance term due to noise that will be controlled by choosing $\tau$ to be suitably large. The second term is the bias which depends on the smoothness of $z^*$ and will be controlled by choosing $\tau$ to be sufficiently small. Hence the optimal choice of $\tau$ will be the one which achieves the right bias-variance trade-off. 

\begin{remark}[Connectedness assumption]\label{rem:connectedness_assum}
\red{The assumption that each $G_k$ is connected enables us to uniquely write the block-centered $z^*$ as $z^* = L^{\dagger} Q x^*$. In case some of the graphs were disconnected, then it is not necessary that $L^{\dagger} Q x^*$ satisfies the smoothness condition in Assumption \ref{assump:smooth}, even though it is a solution of $L z = Qx^*$. It is unclear how to proceed with the analysis in this case, and we leave it for future work. }
\end{remark}
%
%
%

\subsection{Non-evolving case}\label{sec:warm_up_proj}
The following theorem provides a tail bound for the estimation error of the projection method under Assumption \ref{assump:fixed_graph} (\rev{non-evolving} case). 
\begin{theorem}\label{thm:error_proj_fixed_graph}
For any $\delta\in(0,e^{-1})$ and \rev{$\tau\geq 0$}, it holds under Assumptions \ref{assump:smooth} and \ref{assump:fixed_graph} that with probability at least $1-\delta$
\begin{equation*}
\|\est z-z^*\|_2^2\lesssim \rev{S_T\Big(\frac{1}{\tau}\wedge \frac{1}{(n-1)\mu_{T-1}}\Big)}+\alpha_1\sigma^2\log{(1/\delta)} \left( \frac{T+1}{\pi}\sqrt{\frac{\tau}{n-1}}
\sum^{n-1}_{j=1}\frac{1}{\alpha^2_j}+\red{\sum^{n-1}_{j=1}\frac{1}{\alpha^2_j}}\right).\end{equation*}
\end{theorem}
The proof is detailed in Section \ref{sec:proof_thm_proj_evolv}. We now show how to optimize the choice of the regularization parameter $\tau$. 

\paragraph{Choice of $\tau$.} \rev{Consider the case $S_T=0$ first. From Theorem \ref{thm:error_proj_fixed_graph}, it is clear that, in that case, the optimal choice is to set $\tau$ to zero. On the other hand, }notice that Theorem \ref{thm:error_proj_fixed_graph} can be written using asymptotic notation (hiding all the parameters except for $T$ \rev{and $\sigma$}) as 

\begin{equation*}
\|\est z-z^*\|_2^2=O\big(\frac{S_T}{\tau}+\rev{\sigma^2}T\sqrt{\tau}\big),
\end{equation*}
and hence the optimal choice of $\tau$ is given by $\tau=\rev{\sigma^{-\frac43}}\Big(\frac{S_T}{T}\Big)^{\frac23}$. \rev{ Notice that when $\sigma=0$ and $S_T\neq 0$, the optimal choice is $\tau=\infty$, which gives a zero error as expected. In the case $\sigma=S_T=0$, the estimation error is also zero as is obvious from the bound above. Below, we focus on the interesting case where at least one of $\sigma, S_T$ is non-zero.} The following result then follows directly from Theorem~\ref{thm:error_proj_fixed_graph}.  
\begin{corollary}\label{cor:proj_evol}
Choosing $\tau=\rev{\sigma^{-\frac43}}\Big(\frac{S_T}{T}\Big)^{\frac23}$, then for any $\delta\in(0,e^{-1})$, it holds with probability greater than $1-\delta$ that
\begin{equation*}
\|\est z-z^*\|_2^2\lesssim \rev{\sigma^{\frac43}}T^{\frac23}S_T^{\frac13}\Big(1+\frac{\log{(1/\delta)}}{\pi\sqrt{n-1}}\sum^{n-1}_{j=1}\frac{\alpha_1}{\alpha^2_{j}}\Big)+\red{\sigma^2\log{(1/\delta)}\sum^{n-1}_{j=1}\frac{\alpha_1}{\alpha^2_j}}.
\end{equation*}
Consequently, if $S_T = O(T^\gamma)$ for some $\gamma \leq 1$, then $\|\est z-z^*\|_2^2 = O(\rev{\sigma^{\frac43}}T^{\frac{2+\gamma}3}\red{\vee \rev{\sigma^2}})$.
%
In particular, 
\begin{equation*}
\|\tilde z-z^*\|_2^2=
\begin{cases}
O(\rev{\sigma^{\frac43}}T^{\frac23})\text{ if }S_T=O(1) \text{ and } \tau = \rev{\sigma^{-\frac43}}T^{-2/3},\\
O(\rev{\sigma^{\frac43}}T^{\frac13})\text{ if }S_T=O(\frac1T) \text{ and } \tau = \rev{\sigma^{-\frac43}}T^{-4/3}.\\
\end{cases}
\end{equation*}
\end{corollary}

\subsection{Any sequence of connected graphs}
Similar to the analysis for the least-squares approach in Section \ref{sec:evoling_Gt_ls}, we will use the semi-definite bound \eqref{eq:lowner_bound} to pass from the case of non-evolving graphs to the case where the graphs can differ (but are connected). Given the bound in the Loewner order \eqref{eq:lowner_bound}, we see that the role played by the eigenpairs $(\alpha_j,a_j)$ in Section \ref{sec:warm_up_proj} is now replaced by the eigenpairs of the Laplacian of the complete graph. Information about the comparison graphs is also encoded in the scalar $\min_{0\leq k \leq T}{\beta_k}$ (or in $\lamin(L)$), where $\beta_k(n-1)$ is equal to the Fiedler eigenvalue of $L_k$. The following bound for the estimation error is analogous to Theorem \ref{thm:error_proj_fixed_graph} and Corollary \ref{cor:proj_evol}, and formalizes the statement of Theorem \ref{thm:main_result_proj}.
\begin{theorem}\label{thm:error_proj_evol}
Assume that $G_k$ is connected for all $k\in\{0,\dots,T\}$ and let $\lamin(L)$ be the smallest nonzero eigenvalue of $L$. Then for any $\delta\in(0,e^{-1})$ and \rev{$\tau\geq 0$}, it holds under Assumption \ref{assump:smooth} and with probability larger than $1-\delta$ that
\begin{equation}\label{eq:thm_proj_evol_1}
\|\est z-z^*\|_2^2\lesssim \rev{S_T\left(\frac{1}{\tau}\wedge \frac{1}{(n-1)\mu_{T-1}}\right)}+\frac{ (T+1)\sigma^2\log{(1/\delta)}}{\pi}\sqrt{\frac{\tau}{n-1}}
\frac{\norm{L}_2}{\lamin^2(L)}+\red{\sigma^2\log{(1/\delta)}\frac{\norm{L}_2}{\lamin^2(L)}}.
\end{equation}
Choosing $\tau=\rev{\sigma^{-\frac43}}\big(\frac{S_T}{T}\big)^{2/3}$ leads to the bound
\begin{equation}\label{eq:thm_proj_evol_2}
    \|\est z-z^*\|_2^2\lesssim \rev{\sigma^{\frac43}}T^{\frac23}S_T^{\frac13}\Big(1+\frac{\log{(1/\delta)}}{\pi\sqrt{n-1}}\frac{\norm{L}_2}{\lamin^2(L)}\Big)+\red{\sigma^2\log{(1/\delta)}\frac{\norm{L}_2}{\lamin^2(L)}}.
\end{equation}
Thus, if $S_T = O(T^\gamma)$ for some $\gamma \leq \rev{\sigma^2}$, then $\|\tilde z-z^*\|_2^2 = O(\rev{\sigma^{\frac43}}T^{\frac{2+\gamma}3}\red{\vee1})$. In particular, 
\begin{equation*}
\|\est z-z^*\|_2^2=
\begin{cases}
O(\rev{\sigma^{\frac43}}T^{\frac23})\text{ if }S_T=O(1) \text{ and } \tau = \rev{\sigma^{-\frac43}}T^{-2/3},\\
O(\rev{\sigma^{\frac43}}T^{\frac13})\text{ if }S_T=O(\frac1T) \text{ and } \tau =\rev{\sigma^{-\frac43}} T^{-4/3}.\\
\end{cases}
\end{equation*}
\end{theorem}
%
In terms of the notation introduced in Theorem \ref{thm:main_result_proj}, we have %
\begin{align*}
\Psi_{\operatorname{Proj}}\big(n,\sigma, \delta\big)~&=~ \rev{\sigma^{\frac43}}\Big(1~+~\frac{\log{(1/\delta)}}{\pi\sqrt{n-1}}\frac{\norm{L}_2}{\lamin^2(L)}\Big), \\
\red{\Psi'_{\operatorname{Proj}}\big(n,\sigma, \delta\big)}~&=~\red{\sigma^2\log{(1/\delta)}\frac{\norm{L}_2}{\lamin^2(L)}}.
\end{align*}
%

\section{Proofs}\label{sec:proofs}
\subsection{Proof of Lemma \ref{lem:var_term}}\label{sec:proof_lem_var_term}
Using the fact that $\alpha_j\geq \alpha_{n-1}$ for all $j\in\{1,\cdots,n-1\}$, and the explicit form of the non-zero eigenvalues of the path graph on $T+1$ vertices (see for e.g. \cite{brouwer2011spectra})
$$\mu_k=4\sin^2{\frac{(T-k)\pi}{2(T+1)}}, \quad k = 0,\dots,T-1;$$ 
we have that 
\begin{align}
\sum^{T-1}_{k=0}\sum^{n-1}_{j=1}\frac1{(\alpha_j+\lambda (n-1)\mu_k)^2}&\leq (n-1)\sum^{T}_{k=1}\frac{1}{(\alpha_{n-1}+4\lambda(n-1)\sin^2{\frac{k\pi}{2(T+1)}})^2} \label{eq:proof_intermed} 
%
%
\end{align}
The last term in the previous inequality can be regarded as a Riemannian sum, and as such it can be bounded by an integral as follows 
\[\frac{1}{T+1}\sum^{T}_{k=1}\frac{1}{(\alpha_{n-1}+4\lambda(n-1)\sin^2{\frac{k\pi}{2(T+1)}})^2}\leq \int^1_0\frac{dx}{(\alpha_{n-1}+4\lambda (n-1)\sin^2{\frac{\pi x}{2}})^2}.\]
Let us focus on the integral term, for which the following bound holds 
\begin{align*}
\int^1_0\frac{dx}{(\alpha_{n-1}+4\lambda (n-1)\sin^2{\frac{\pi x}{2}})^2}&\stackrel{(1)}{=}2\int^{\frac12}_0\frac{dx}{(\alpha_{n-1}+4\lambda (n-1)\sin^2{\pi x})^2}\\
&\stackrel{(2)}{\leq} 2\int^{\frac12}_0\frac{dx}{(\alpha_{n-1}+\lambda (n-1)\pi^2x^2)^2}\\
&\stackrel{(3)}{=}\frac{2}{\pi\alpha^{3/2}_{n-1}\sqrt{\lambda(n-1)}}\operatorname{I}\Big(\frac{\pi}2\sqrt{\frac{\lambda(n-1)}{\alpha_{n-1}}}\Big),
\end{align*}
where $\operatorname{I}(t)=\int^t_0\frac{du}{(1+u^2)^2}$. Equality $(1)$ comes from a change of variables. In $(2)$, we used that $\sin x\geq x/2$ for all $x\in[0,\pi/2]$ and $(3)$ results from the change of variable $u=\sqrt{\frac{\lambda(n-1)}{\alpha_{n-1}}}\pi x$. By an elementary calculation, we verify that $\operatorname{I}(t)=\frac{\arctan t}{2}+\frac{t}{2(1+t^2)}$. Hence, 
\begin{align*}
\int^1_0 \frac{dx}{(\alpha_{n-1}+4\lambda (n-1)\sin^2\pi x)^2}
&\leq \frac{\rev{1}}{\alpha^{3/2}_{n-1}}\frac{\arctan(\frac{\pi}2\sqrt{\frac{\lambda(n-1)}{\alpha_{n-1}}})}{\pi\sqrt{\lambda(n-1)}} +\frac{1}{\rev{2(}\alpha^2_{n-1}+\frac{\pi^2}2 \alpha_{n-1}(n-1)\lambda\rev{)}}\\
&\lesssim \frac{1}{\alpha^{3/2}_n\sqrt{(n-1)\lambda}}+\frac{1}{\alpha^2_{n-1}+\frac{\pi^2}{\rev{4}} \alpha_{n-1}(n-1)\lambda},
\end{align*}
where the last inequality follows from $|\arctan(x)|\leq \frac{\pi}2$ for all $x\in \matR$. \rev{From this, we deduce that \begin{align*} 
\sum^{T-1}_{k=0}\sum^{n-1}_{j=1}\frac1{(\alpha_j+\lambda (n-1)\mu_k)^2}&\lesssim \frac{T\sqrt{n-1}}{\alpha_{n-1}^{3/2}\sqrt{\lambda}}+ \frac{T(n-1)}{\alpha_{n-1}^2+\alpha_{n-1}(n-1)\lambda}\\
&= T(n-1)\left(\frac{1}{\alpha_{n-1}\sqrt{\alpha_{n-1}(n-1)\lambda}}+ \frac{1}{\alpha_{n-1}^2+\alpha_{n-1}(n-1)\lambda}\right)\\
&\lesssim\frac{T(n-1)}{\alpha_{n-1}\sqrt{\alpha_{n-1}(n-1)\lambda}}=\frac{T\sqrt{n-1}}{\alpha_{n-1}^{3/2}\sqrt{\lambda}}.
\end{align*}}

\subsection{Proof of Proposition \ref{thm:error_ls_fixed_graph}} \label{sec:prop_fix_graph_pen_error}
The proof of Proposition \ref{thm:error_ls_fixed_graph} follows from Lemmas \ref{lem:ls_var_term} and \ref{lem:ls_bias_term} below. 

\begin{lemma}[Bound on the variance term]\label{lem:ls_var_term}
For any $\delta\in (0,e^{-1})$, it holds with probability larger than $1-\delta$ that
\begin{equation*}
    \|L^\dagger (\lambda)Q(y-x^*)\|_2^2\leq \sigma^2 \alpha_1 (1+4\log(1/\delta))\Big(\sum^T_{k=0}\frac{1}{(\alpha_j+\lambda(n-1)\mu_k)^2}+\sum^{n-1}_{j=1}\frac{1}{\alpha_j^2}\Big).
\end{equation*}
\end{lemma}
\begin{proof}
 Define $\Sigma:=(L^\dagger (\lambda)Q)^\top L^\dagger (\lambda)Q=Q^\top L^{\dagger 2} (\lambda)Q$. Using \rev{\cite[Thm.2.1]{hsu2012tail}}, it holds for any $c>1$
 \begin{align}
    \|L^\dagger (\lambda)Q(y-x^*)\|_2^2&\rev{\leq \sigma^2(\Tr(\Sigma)+2\sqrt{\Tr(\Sigma^2)c}+2\|\Sigma\|_2c)}\nonumber\\ \label{eq:ls_concen_var}
    &\leq \sigma^2(1+4c)\Tr(\Sigma)
 \end{align}
 with probability $1-e^{-c}$ \rev{(using the fact that $\Tr(\Sigma^2) \leq \|\Sigma\|_2 \Tr(\Sigma) \leq (\Tr(\Sigma))^2$)}. On the other hand, 
 
\begin{align*}
    \Tr(\Sigma)&=\Tr(QQ^\top L^{\dagger 2}(\lambda)) \leq \|QQ^\top\|_2\Tr(L^{\dagger 2}(\lambda))
\end{align*}
where we used the fact that given symmetric p.s.d matrices $A,B$ it holds $\Tr(AB)\leq \|A\|_2\Tr(B)$. From the spectral decomposition of $L(\lambda)$ (given by \eqref{eq:decomp_L_lambda}) we deduce that 
\begin{equation*}
    \Tr(L{^{\dagger 2}} (\lambda))=\sum^{T-1}_{k=0}\sum^{n-1}_{j=1}\frac{1}{(\alpha_j+\lambda(n-1)\mu_k)^2}+\sum^{n-1}_{j=1}\frac{1}{\alpha_j^2}.
\end{equation*}
Using $\norm{QQ^\top}_2 = \alpha_1$ and $c = \log(1/\delta)$ concludes the proof.
\end{proof}

\begin{lemma}[Bound on the bias term]\label{lem:ls_bias_term}
We have 
\begin{equation*}
    \lambda^2\|L^\dagger(\lambda)E^\top Ez^*\|_2^2\leq \Big(\frac{1}{\alpha^2_{n-1}}\vee 1\Big)\lambda S_T.
\end{equation*}
\end{lemma}
\begin{proof}
From \eqref{eq:decomp_spec} and \eqref{eq:decomp_L_lambda} it is easy to see that 
\[\lambda^2\|L^\dagger(\lambda)E^\top  Ez^*\|_2^2=\sum^{T-1}_{k=0}\sum^{n-1}_{j=1}\Big(\frac{\lambda(n-1)\mu_k}{\alpha_j+\lambda(n-1)\mu_k}\Big)^2\langle z^*,u_k\otimes a_j\rangle ^2.\]
On the other hand, the smoothness condition in Assumption \ref{assump:smooth} can be written as
\[
\|Ez^*\|_2^2=\sum^{T-1}_{k=0}\sum^{n-1}_{j=1}(n-1)\mu_k\langle z^*,u_k\otimes a_j\rangle ^2\leq S_T.
\]
Defining $b_k:=\sum^{n-1}_{j=1}\langle z^*,u_k\otimes a_j\rangle ^2$ and using that $\alpha_1\geq\cdots\geq\alpha_{n-1}>0$ we obtain 
\begin{align*}
    \lambda^2\|L^\dagger(\lambda)E^\top  Ez^*\|_2^2&\leq\sum^{T-1}_{k=0}\Big(\frac{\lambda(n-1)\mu_k}{\alpha_{n-1}+\lambda(n-1)\mu_k}\Big)^2 b_k,\\
    \lambda\|Ez^*\|_2^2&=\sum^{T-1}_{k=0}\lambda(n-1)\mu_k b_k\leq \lambda S_T.
\end{align*}
We now conclude using the following elementary fact with $c=\lambda (n-1)\mu_k$ and $d=\alpha_{n-1}$. 
\begin{claim}\label{claim:elementary}
Let $c,d$ be strictly positive real numbers, then \[\Big(\frac{c}{d+c}\Big)^2\leq \Big(\frac{1}{d^2}\vee 1\Big)c.\]
\end{claim}
\begin{proof}
 We have two cases. For $c\geq 1$, we have $\Big(\frac{c}{d+c}\Big)^2<1$ and $\Big(\frac{1}{d^2}\vee 1\Big)c\geq 1$, so the inequality is verified. In the case $c<1$, we have \[\Big(\frac{c}{d+c}\Big)^2\leq \frac{c^2}{d^2}\leq \frac{c}{d^2}\leq \Big(\frac{1}{d^2}\vee 1\Big)c.\]
\end{proof}

\end{proof}

\subsection{Proof of Proposition \ref{thm:error_ls_evol_graph}}\label{sec:proof_prop2}
Using \eqref{eq:loewner_bd_3}, the proof of Proposition \ref{thm:error_ls_evol_graph} goes along the same lines as that of Proposition \ref{thm:error_ls_fixed_graph}. It follows directly from Lemmas \ref{lem:bound_bias_non_evol} (the first statement) and \ref{lem:bound_var_non_evol} below, which offer a control of the bias and the variance term,  respectively.


\begin{lemma}\label{lem:bound_bias_non_evol}
\red{
The following is true.}
\begin{enumerate}
\item \red{It holds that 
\begin{equation*}
    \lambda^2 \|L^\dagger(\lambda)E^TEz^*\|^{2}_2\lesssim \frac{\lambda^2 n}{\lamin^2(L)} S_T.
\end{equation*}}

\item \red{If $G_k$ is connected for all $k = 0,\cdots,T$ and $L$ satisfies Assumption \ref{assump:technical_lower_bound}, then for all $\lambda > 0$ we have 
\begin{equation*}
    \lambda^2 \|L^\dagger(\lambda)E^TEz^*\|^{2}_2\lesssim \frac{\okap}{\okap-1}\Big(\frac{1}{\lamin^2(L)}\vee 1\Big) \lambda S_T
\end{equation*}
where \[\okap:=\max\Big\{\kappa>1 \text{ s.t }\frac1{\kappa}\big(L^2+\lambda^2(E^TE)^2\big)+\lambda(E^TEL+LE^TE)\succcurlyeq 0\Big\}.\]}
\end{enumerate}
\end{lemma}
\begin{proof}
\red{
We have 
\begin{align*}
    \lambda^2 \|L^\dagger(\lambda)E^TEz^*\|^{2}_2&\leq
    \lambda^2\|L^\dagger(\lambda)E^T\|^2_2\|Ez^*\|_2^2\\
    &\leq \lambda^2 \|L^\dagger(\lambda)\|^2_2\|E^T\|^2_2 \|Ez^*\|_2^2 \\
    &\lesssim \frac{\lambda^2 n}{\lamin^2(L)}S_T
\end{align*}
where in the first inequality we used the submultiplicativity of the operator norm and in the last step we used that $\|E^T\|^2_2 \lesssim n$, $\|Ez^*\|^2_2\leq S_T$ by the smoothness assumption, and $\|L^\dagger(\lambda)\|_2^2 \leq \frac1{\lamin^2(L)}$. This proves the first statement. } 

\red{To prove the second part, we assume that Assumption \ref{assump:technical_lower_bound} holds and $\overline{\kappa}$ is the largest value for which it holds. From this it follows that 
\begin{align*}
(L+\lambda E^TE)^2&\succcurlyeq \frac{\okap-1}{\okap}\big(L^2+\lambda^2(E^TE)^2\big)\\
&\succcurlyeq \frac{\okap-1}{\okap}\left(\frac{\lamin^2(L)}{(n-1)^2}\tilde{L}^2+\lambda^2(E^TE)^2\right),
\end{align*}
where in the last inequality we used \eqref{eq:lowner_bound} with $\min_{0\leq k\leq T}{\beta_k} = \frac{\lamin(L)}{n-1}$ and the fact that the eigenvectors of $L$ and $\tilde{L}$ are aligned \rev{(because $\tilde{L}=I_{T+1}\otimes \Lcom$ and the eigenvector of $\Lcom$ can be aligned with those of any Laplacian)}. Since $\lambda > 0$, then by recalling Lemma \ref{lem:rank_Laplacian_reg}, one can verify that  $(L+\lambda E^TE)^2$ and $\tilde{L}^2+\lambda^2 (E^TE)^2$ have the same null space. This in turn  implies \rev{(using Lemma \ref{lem:loewner_pseudo})} that 
\[
\frac{\okap}{\okap-1}\left(\frac{\lamin^2(L)}{(n-1)^2}\tilde{L}^2+\lambda^2(E^TE)^2\right)^\dagger\succcurlyeq(L+\lambda E^TE)^{\dagger 2}=L^{\dagger 2}(\lambda).
\]
Now arguing similarly to the proof of Lemma \ref{lem:ls_bias_term}, we have 
\begin{align*}
    \|L^\dagger(\lambda)E^TEz^*\|^2_2&=z^{*T}E^TEL^{\dagger 2} (\lambda)E^TEz^*\\
    &\leq \frac{\okap}{\okap-1}z^{* T}E^TE\left(\frac{\lamin^2(L)}{(n-1)^2}\tilde{L}^2+\lambda^2(E^TE)^2\right)^\dagger E^TEz^*,\\
    &\stackrel{(1)}{=}\frac{\okap}{\okap-1} \sum^{T-1}_{k=0}\sum^{n-1}_{j=1}\frac{\lambda^2(n-1)^2\mu^2_k}{\lamin^2(L)+\lambda^2(n-1)^2\mu^2_k}\langle z^*,u_k\otimes a_j\rangle ^2\\
    &\stackrel{(2)}{\leq} \frac{\okap}{\okap-1}\Big(\frac{1}{\lamin^2(L)}\vee 1\Big)\sum^{n-1}_{j=1}\lambda(n-1)\mu_k\langle z^*,u_k\otimes a_j\rangle ^2,\\
    &\stackrel{(3)}{\leq} \frac{\okap}{\okap-1}\Big(\frac{1}{\lamin^2(L)}\vee 1\Big) \lambda S_T
\end{align*}
where in $(1)$ we used the known spectral expansion of $\tilde{L}$ and $E^TE$, in $(2)$ we used Claim 
\ref{claim:elementary} (used in the proof of Lemma \ref{lem:ls_bias_term}) with the fact that $\frac{c^2}{c^2+d^2}\leq \big(\frac{c}{c+d}\big)^2$ for $c,d>0$, and in $(3)$ we used the smoothness assumption. }
\end{proof}
%
When the comparison graphs are not necessarily the same, Lemma \ref{lem:ls_var_term} is replaced by the following lemma. 
\begin{lemma}\label{lem:bound_var_non_evol}
For any $\delta\in (0,e^{-1})$, it holds with probability greater than $1-\delta$ that 
\begin{equation*}
    \|L^{\dagger}(\lambda)Q(y-x^*)\|_2^2\leq \sigma^2\|L\|_2(1+4\log(1/\delta))\Big(\sum^{T-1}_{k=0}\frac{n-1}{(\lamin(L)+(n-1)\lambda\mu_k)^2}+\frac{n-1}{\lamin^2(L)}\Big).
\end{equation*}
\end{lemma}
\begin{proof}
 Reasoning as in the proof of Lemma \ref{lem:ls_var_term} (using \rev{\cite[Thm. 2.1]{hsu2012tail}}) we obtain with probability at least $1-\delta$ that
 \begin{equation}\label{eq:ls_concen_evol_var}
     \|L^\dagger Q(y-x^*)\|^2\leq \sigma^2 (1+4\log(1/\delta))\Tr\big(L L^{\dagger 2}(\lambda)\big).
 \end{equation}
Recall from \eqref{eq:loewner_bd_3} that $L(\lambda) \succcurlyeq \tilde{L}(\lambda)$, hence we deduce that 
\begin{align*}
\Tr\big(L L^{\dagger 2}(\lambda)\big) \leq \|L\|_2\Tr\big({L}^{\dagger 2}(\lambda)\big) \leq \|L\|_2\Tr\big(\tilde{L}^{\dagger 2}(\lambda)\big). 
\end{align*}
Since $\tilde{L}^{\dagger}(\lambda)=\big(\frac{1}{n-1}\lamin(L)I_{T+1}+\lambda MM^\top\big)\otimes \Lcom$, it is easy to verify that the nonzero eigenvalues of $\tilde{L}(\lambda)$ correspond to the set $\{\lamin(L)+\lambda\mu_k (n-1)\}^{T-1}_{k=0}\cup \{\lamin(L)\}$, and each eigenvalue has multiplicity $n-1$. Hence,
 \begin{equation*}
     \Tr\big(\tilde{L}^{\dagger 2}(\lambda)\big)=\sum^{T-1}_{k=0}\frac{n-1}{(\lamin(L)+\lambda\mu_k(n-1))^2}+\frac{n-1}{\lamin^2(L)}.
 \end{equation*}
Plugging this into \eqref{eq:ls_concen_evol_var}, the result follows. 
\end{proof}

\subsection{Proof of Theorem \ref{thm:error_proj_fixed_graph}}
\label{sec:proof_thm_proj_evolv}
The proof of Theorem \ref{thm:error_proj_fixed_graph} follows directly from the two lemmas below.
\begin{lemma}\label{lem:proj_smooth_part}
If $z^*$ is such that $\|Ez^*\|_2^2\leq S_T$, and \rev{$\tau\geq 0$} then $\|\hfreqproj z^*\|_2^2\leq \rev{S_T(\frac{1}{\tau}\wedge\frac{1}{(n-1)\mu_{T-1}})}$.
\end{lemma}
\begin{proof}
By Parseval's identity,
\[\|\hfreqproj z^*\|_2^2=\sum_{(k,j)\in \calH_\tau}\langle z^*,u_k\otimes a_j\rangle ^2.\]
On the other hand, 
\begin{align*}
\sum_{(k,j)\in \calH_\tau}(n-1)\mu_k\langle z^*,u_k\otimes a_j\rangle^2&\leq \sum^{T-1}_{k=0}\sum^{n-1}_{j=1}(n-1)\mu_k\langle z^*,u_k\otimes a_j\rangle^2 = \|Ez^*\|_2^2 \leq S_T,
\end{align*}
where we used the fact $\mu_k\geq 0$ in the first inequality, and the assumption $\|Ez^*\|_2^2\leq S_T$ in the second inequality. \rev{We first assume that $\tau>0$}. Since $(n-1)\mu_k \geq \tau$  for all $(k,j)\in \calH_\tau$, we deduce that\[\|\hfreqproj z^*\|_2^2\leq \frac{S_T}{\tau}.\]
\rev{On the other hand, we have the following inequalities
\[(n-1)\mu_{T-1}\|\hfreqproj z^*\|^2\leq \|E \hfreqproj z^*\|^2\leq \|E z^*\|^2\leq S_T,\]
from which the lemma follows.
}

\end{proof}

\begin{lemma}\label{lem:proj_noisy_part}
For any $\delta \in (0,e^{-1})$, it holds with probability greater than $1-\delta$
\begin{equation*}
\|\lfreqproj L^\dagger Q(y-x^*)\|_2^2\leq 2\sigma^2 \alpha_1 (1+4\log{(1/\delta)})\left(\frac{T+1}{\pi} \sqrt{\frac{\tau}{n-1}}\sum^{n-1}_{j=1}\frac{1}{\alpha^2_j}\red{+ \sum^{n-1}_{j=1}\frac1{\alpha^2_j}}\right).
\end{equation*}
\end{lemma}
\begin{proof}
Define \[\Sigma:=(\lfreqproj L^\dagger Q)^\top \lfreqproj L^\dagger Q=Q^\top L^\dagger \lfreqproj L^\dagger Q.\]
%
Using \rev{\cite[Thm.2.1]{hsu2012tail}}, we have for all $c>1$ and with probability larger than $1-e^{-c}$ that
\begin{equation}\label{eq:tail_bound_proj}
\|\lfreqproj L^\dagger Q(y-x^*)\|_2^2\leq \sigma^2(1+4c)\Tr(\Sigma).
\end{equation}
By the cyclic property of the trace, it holds that
\begin{equation} \label{eq:tmp_bd_12}
\Tr(\Sigma)=\Tr (QQ^\top L^\dagger \lfreqproj L^\dagger)\leq \|QQ^\top\|_2\Tr(L^\dagger \lfreqproj L^\dagger),
\end{equation}
using the same property of p.s.d matrices as in the proof of Lemma \ref{lem:ls_var_term}. On the other hand,  
\begin{align*}
\lfreqproj = \sum_{(k,j)\in \calL_\tau} (u_k\otimes a_j)(u_k\otimes a_j)^\top \text{  and  }
L^\dagger &=\sum^{\red{T}}_{k=0}\sum^{n-1}_{j=1}\frac{1}{\alpha_j}(u_k\otimes a_j)(u_k\otimes a_j)^\top.
\end{align*}
Recall that $\calL_\tau=\{(k,j)\in\{0,\cdots, T\}\times [n]\text{ s.t }\lambda_j\mu_k\rev{\leq}\tau\}$ can be written as $\calL_1 \cup \calL_2$, where
\begin{align*}
    \calL_1 &= \{k=0,\cdots, T-1 \text{ s.t }(n-1)\mu_k\rev{\leq}\tau\} \times [n-1], \\
    \calL_2 &= \{(k,j)\in\{0,\cdots, T\}\times [n]\text{ s.t }\lambda_j\mu_k=0\} = (\{T\}\times [n]) \cup (\set{0,\dots,T} \times \set{n}).
\end{align*}
Thus, 
\begin{equation*}
\lfreqproj {L^{\dagger 2}} =\sum_{(k,j)\in \calL_1}\frac{1}{\alpha^2_j} (u_k\otimes a_j)(u_k\otimes a_j)^\top \red{+\sum^{n-1}_{j=1}\frac{1}{\alpha_j^2}(\ones_T\otimes a_j)(\ones_T\otimes a_j)^T}.
\end{equation*}
and we obtain the expression
\begin{equation} \label{eq:tmp_bd_11}
\Tr(\lfreqproj {L^{\dagger 2}})=\Big(|\{k=0,\cdots, T-1 \text{ s.t }(n-1)\mu_k\rev{\leq}\tau\}|\red{+1}\Big)\times \sum^{n-1}_{j=1}\frac1{\alpha^2_j}.
\end{equation}
Recall that $\mu_k = 4\sin^2\frac{(T-k)\pi}{2(T+1)}$ for $k=0,\dots,T-1$ and define the set
\begin{align*}
\calA:=\{k=0,\cdots, T-1 \text{ s.t }4(n-1)\sin^2\frac{(T-k)\pi}{2(T+1)} \rev{ \leq } \tau\}.
\end{align*}
Now since $\sin x \geq x/2$ for $x \in [0,\pi/2]$, we know that 
\begin{equation*}
\sin^2\frac{(T-k)\pi}{2(T+1)} \geq  \frac{(T-k)^2\pi^2}{16(T+1)^2}.
\end{equation*}
This in turn implies that $\abs{\calA}$ can be bounded as
\begin{align*}
  \abs{\calA} &\leq  \abs{\set{k=1,\dots,T: 4(n-1)\frac{\pi^2 k^2}{16(T+1)^2} \rev{\leq} \tau}}  \\
  &\leq \min\set{T, \frac{2(T+1)}{\pi}\sqrt{\frac{\tau}{n-1}}} \\
  &\leq 2\left(\frac{T+1}{\pi}\right)\sqrt{\frac{\tau}{n-1}}.
\end{align*}
Using this in \eqref{eq:tmp_bd_11}, we obtain
\begin{equation*}
Tr(\lfreqproj {L^{\dagger2}})\leq 2\left(\frac{T+1}{\pi} \right) \sqrt{\frac{\tau}{n-1}}\sum^{n-1}_{j=1}\frac1{\alpha^2_j}\red{+\sum^{n-1}_{j=1}\frac1{\alpha^2_j}}.
\end{equation*}
Combining this with \eqref{eq:tail_bound_proj} and \eqref{eq:tmp_bd_12}, the result follows by taking $c=\log{(1/\delta)}$.
\end{proof}

\subsection{Proof of Theorem \ref{thm:error_proj_evol}}

We first observe that Lemma \ref{lem:proj_smooth_part} works under Assumption \ref{assump:smooth} and hence applies verbatim in this context. We now formulate a result analogous to Lemma \ref{lem:proj_noisy_part} dropping the assumption that all the graphs are the same and using the semi-definite bound \eqref{eq:lowner_bound}.

\begin{lemma}\label{lem:proj_noisy_part2}
Assume that $G_k$ is connected for all $k = 0,\dots,T$ and let $\beta_k(n-1)$ be the Fiedler eigenvalue of $L_k$. For any $\delta \in (0,e^{-1})$, it holds with probability larger than $1-\delta$
\begin{equation*}
\|\lfreqproj L^\dagger Q(y-x^*)\|_2^2\leq 2\sigma^2\frac{\norm{L}_2}{\min_{0\leq k\leq T}{\beta_k^2}}(1+4\log{(1/\delta)})\left(\frac{T+1}{\pi(n-1)} \sqrt{\frac{\tau}{n-1}}+\red{\frac 1{n-1}}\right).
\end{equation*}
\end{lemma}
\begin{proof}
 Let $\Sigma:=(\lfreqproj L^\dagger Q)^\top \lfreqproj L^\dagger Q=Q^\top L^\dagger \lfreqproj L^\dagger Q$ \rev{and recall that $y-x^*$ is  subgaussian with parameter $\sigma$}. Using \rev{\cite[Thm.2.1]{hsu2012tail}} we obtain for any $c>1$ that with probability larger than $1-e^{-c}$
 \begin{align}
 \|\lfreqproj L^\dagger Q(y-x^*)\|_2^2&\rev{\leq \sigma^2(\Tr(\Sigma)+2\sqrt{\Tr(\Sigma^2)c}+2\|\Sigma\|_2c)}\nonumber\\
 \label{eq:conc_lemma_proj1}
 &\leq \sigma^2(1+4c)\Tr(\Sigma)\leq \sigma^2(1+4c)\|QQ^\top\|_2\Tr(L^\dagger \lfreqproj L^\dagger),
 \end{align}
 \rev{where to pass from the first line to the second line we used the fact that $c>1$ and $\Tr(\Sigma^2) \leq \|\Sigma\|_2 \Tr(\Sigma) \leq (\Tr(\Sigma))^2$.}
 Notice that up until this point, the proof does not change with respect to the proof of Lemma \ref{lem:proj_noisy_part}. Recall  from \eqref{eq:lowner_bound} that $\tilde L=I_{T+1}\otimes \Lcom$  and $L \succcurlyeq (\min_{0\leq k\leq T}{\beta_k}) \tilde L$. Since the eigenvectors of $L, \tilde L$ are aligned \rev{(because the eigenvector of $\Lcom$ can be aligned with those of any Laplacian)}, it follows that  
 \begin{align} \label{eq:low_bd_13_tmp}
   \frac1{\min_{0\leq k\leq T}{\beta_k^2}} \tilde{L}^{\dagger 2} \succcurlyeq  L^{\dagger 2}. 
 \end{align}
Using \eqref{eq:low_bd_13_tmp} and the fact that $\lfreqproj$ is a projection matrix, we obtain 
\begin{align} \label{eq:bound_trace_lemma}
    \Tr(L^\dagger \lfreqproj L^\dagger)&=\Tr(\lfreqproj {L^{\dagger 2}})\leq \frac1{\min_{0\leq k\leq T}{\beta_k^2}}\Tr(\lfreqproj {\tilde{L}^{\dagger 2}}).
\end{align}
\red{where the second inequality follows from \eqref{eq:low_bd_13_tmp} and the fact that $\Tr(AB) \geq 0$ if $A,B \succcurlyeq 0$.}
%
%
We also have that
\begin{equation*}
    \lfreqproj {\tilde{L}^{\dagger 2}}=\sum_{(k,j)\in \calL_1}\frac{1}{(n-1)^2} (u_k\otimes a_j)(u_k\otimes a_j)^\top
\end{equation*}
hence proceeding in a manner similar to the proof of Theorem \ref{thm:error_proj_fixed_graph} leads to the bound 
\begin{align*}
    \Tr(\lfreqproj {\tilde{L}^{\dagger 2}}) &=  \big(|\{k=0,\cdots, T-1 \text{ s.t }(n-1)\mu_k\rev{\leq}\tau\}|\red{ + 1}\big)\times \frac{1}{n-1}\\
    &\leq 2\frac{T+1}{\pi(n-1)}\sqrt{\frac{\tau}{n-1}}\red{+\frac{1}{n-1}}.
\end{align*}
Combining this with \eqref{eq:conc_lemma_proj1} and \eqref{eq:bound_trace_lemma}, the result follows. 
\end{proof}

\begin{proof}[Proof of Theorem \ref{thm:error_proj_evol}]
 The first inequality of Theorem \ref{thm:error_proj_evol} follows directly from Lemmas \ref{lem:proj_smooth_part} and \ref{lem:proj_noisy_part2}. Inequality \eqref{eq:thm_proj_evol_1} follows by choosing $\tau$ as in Section \ref{sec:warm_up_proj}. The rest of the proof is a direct consequence of the previous result.
\end{proof}

\section{Experiments}\label{sec:experiments}

We now empirically evaluate the performance of the proposed methods on both synthetic and real data. The code for all the experiments is available at \url{https://github.com/karle-eglantine/Dynamic_TranSync}.
\subsection{Synthetic data} \label{sec:synth_exps}

\paragraph{Generate the ground truth $z^*$.} The first step is to generate the true weight vector $z^*$ such that it satisfies Assumption \ref{assump:smooth}, for different regimes of $S_T$. Let us recall that this assumption can be written as
\[ \norm{Ez^*}_2^2 = z^{*\top} E^\top Ez^*\leq S_T\]
which implies that $z^*$ lies close to the null space of $E^\top E$, as quantified by $S_T$. Hence, one possibility to simulate $z^*$ is to project any vector onto a space $\calV_{\varepsilon}$ generated by the eigenvectors associated with the smallest eigenvalues of $E^\top E$. Likewise, $\calV_{\varepsilon}$ contains the null space of $E^\top E$ and a few more eigenvectors, depending on $S_T$. For $\varepsilon > 0$, let us define 
\[ \calL_{\varepsilon} := \set{(k,j) \in\set{0,\dots,T}\times [n] \text{ s.t. } \lambda_j\mu_k \rev{\leq} \varepsilon},\]
then $\calV_{\varepsilon}=\spn\{u_k\otimes a_j\}_{(k,j)\in \calL_\tau}$.  Using similar tools as in the proof of Theorem \ref{thm:error_proj_fixed_graph}, it can be shown that
\[ |\calL_\varepsilon| \leq T+n+\frac{\sqrt{(n-1)\varepsilon}(T+1)}{\pi}.\]
Hence, we can compute a value of $\varepsilon$ such that any vector belonging to $\calV_{\varepsilon}$ satisfies Assumption \ref{assump:smooth}. More precisely, choosing $\varepsilon = \left(\frac{\pi S_T}{(T+1)\sqrt{n-1}}\right)^{2/3}$, it then holds that for any $z \in \matR^{n(T+1)}$ such that $\norm{z}_2 = 1$,
\[ \norm{E P_{\calV_\varepsilon} z}_2^2 \leq S_T,\]
where $P_{\calV_\varepsilon}$ denotes the projection matrix onto the space $\calV_\varepsilon$.
The synthetic data is then generated as follows.
\begin{enumerate}
    \item Generate $z \sim \calN(0,I_{n(T+1)})$ and normalize it such that $\norm{z}_2 = 1$.
    \item Compute $k_\varepsilon = \lceil T+n + \left(\frac{T+1}{(n-1)\pi}\right)^{2/3}S_T^{1/3}\rceil$ eigenvectors of $E^\top E$, corresponding to its $k_\varepsilon$ smallest eigenvalues. 
    \item Compute $z^* = P_{\calV_\varepsilon} z$.
    \item Center each block $z^*_k$, $k=0,\dots, T$.
\end{enumerate}

\paragraph{Generate the observation data.} The observations consist of comparison graphs $G_t$ for $t\in\calT$ and the associated measurements $y_{ij}(t)$ for each edge  $\set{i,j}\in\calE_t$.
\begin{enumerate}
    \item For each $t\in\calT$,  $G_t$ is generated as an Erdös-Renyi graph $G(n,p(t))$, where the probability $p(t)$ is chosen appropriately, as described in Figs. \ref{fig:plots_transync}, \ref{fig:plots_btl} and \ref{fig:plots_transync_ls}. Meaningful recovery of $z^*$ is only possible if the union of all the graphs is connected, hence we ensure the connectivity of $G_U = \left(n,\cup_{t\in\calT} \calE_t\right)$.
    
    \item For all $t\in\calT$ and $\set{i,j}\in\calE_t$, we generate a noisy measurement $y_{ij}(t)$ of the strengths difference $z^*_{t,i}-z^*_{t,j}$ as in \eqref{eq:transync_model}, using a standard Gaussian noise.
\end{enumerate}

Once the true weights and the observations are generated, one can easily implement our methods using traditional least-squares solvers. In our experiments, we use the least-squares solver \texttt{lsqr} from the \texttt{scipy} package of Python. Note that in both of our methods, there is a tuning parameter -- $\lambda$ for the Penalized Least-Squares (denoted as DLS later) and $\tau$ for the Projection method (denoted DProj). Throughout, we choose $\lambda = (T/S_T)^{2/3}$ and $\tau = (S_T/T)^{2/3}$.

\paragraph{Dynamic BTL set up.} As noted in Remark \ref{rem:btl_setup}, the Dynamic TranSync model is linked to the Dynamic BTL model. Hence, we will also test numerically the performances of our methods on synthetic data generated according to the BTL model. The ground truth $w^* = \exp(z^*)$ is generated similarly as in the Dynamic TranSync setup, as well as the observation graphs. However the measurements $y_{ij}(t)$ are now generated according to the Dynamic BTL model (see \cite{karle2021dynamic} for more details). In this particular setup, we will compare the performance of our methods with two other approaches that focus on dynamic ranking for the BTL model, namely, the Maximum-Likelihood Estimation (MLE) method \cite{bong2020nonparametric} and the Dynamic Rank Centrality (DRC) method  \cite{karle2021dynamic}. These methods have shown optimal results when the strength of each item is a Lipschitz function of time\rev{, which translates to $S_T = \frac{1}T$ in our set up. More generally, the regime of interest in our smoothness assumption is $S_T = o(T)$.} 
%

\begin{figure}[h!]
\centering
\begin{subfigure}{.5\textwidth}
  \centering
  \includegraphics[width=\linewidth]{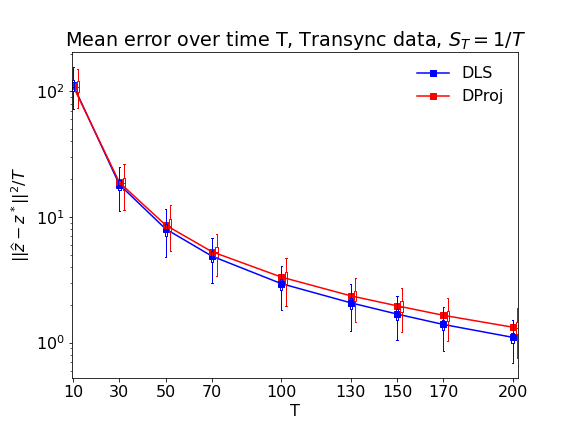}
  \caption{Smoothness $S_T = \frac{1}{T}$}
  \label{fig:transync_alpha_ST_1}
\end{subfigure}%
\begin{subfigure}{.5\textwidth}
  \centering
  \includegraphics[width=\linewidth]{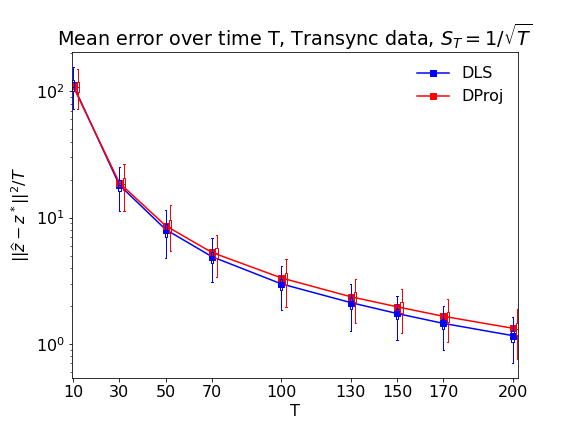}
  \caption{Smoothness $S_T = \frac{1}{\sqrt{T}}$}
  \label{fig:transync_alpha_ST_05}
\end{subfigure}%
\hfill
\begin{subfigure}{.5\textwidth}
  \centering
  \includegraphics[width=\linewidth]{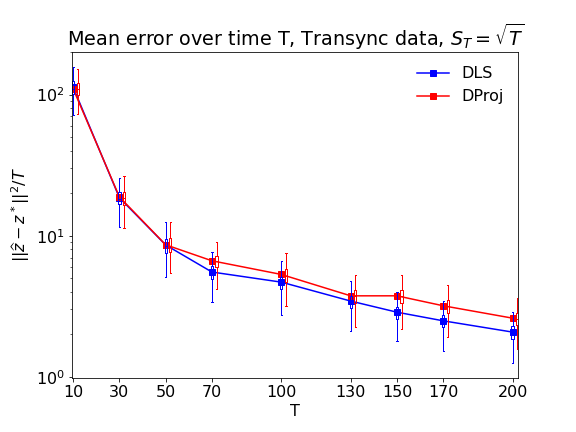}
  \caption{Smoothness $S_T = \sqrt{T}$}
  \label{fig:transync_alpha_ST_m05}
\end{subfigure}%
\begin{subfigure}{.5\textwidth}
  \centering
  \includegraphics[width=\linewidth]{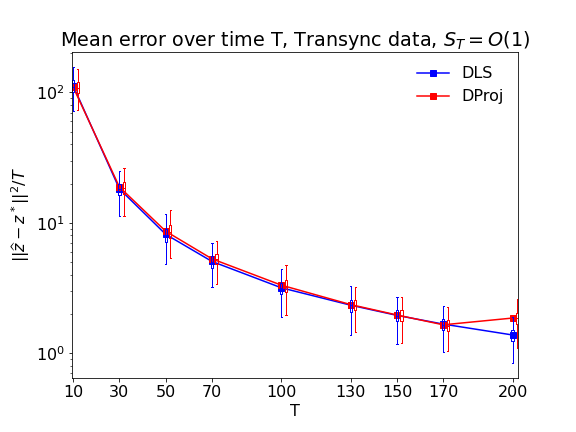}
  \caption{Smoothness $S_T = O(1)$}
  \label{fig:transync_alpha_ST_0}
\end{subfigure}%
\caption{\rev{MSE versus $T$ for DLS and DProj when the data is generated according to the Dynamic TranSync model for $n=100$ and graphs are generated as $G(n,p(t))$ with $p(t)$ chosen randomly between $\frac{1}{n}$ and $\frac{\log(n)}{n}$. The results are averaged over the grid $\calT$ as well as $40$ Monte Carlo runs.}}
\label{fig:plots_transync}
\end{figure}

\paragraph{Results.} The results are summarized below.
\begin{enumerate}
\item In Figure \ref{fig:plots_transync}, we consider $n = 100$ items and $T$ ranging from $10$ to $200$, where the data is generated according to the Dynamic TranSync model. The estimation errors are averaged over $40$  bootstrap simulations and plotted for different regimes of smoothness $S_T$. We observe that in every case, DLS and DProj methods give very similar results. 
As expected, the MSE decreases to zero as $T$ increases in every smoothness regime. Note that the variance of the error, represented by the vertical bars, is also decreasing with $T$. 

\item In Figure \ref{fig:plots_btl}, we repeat the same experiments using data simulated according to the Dynamic BTL Model, and compare our results with the DRC and MLE methods for $S_T = \frac{1}{T}$ and $S_T = \frac{1}{\sqrt{T}}$. For both smoothness choices, DLS seems to be the best performing method as its MSE goes to zero at the fastest rate.

\item As a sanity check, we also show in Figure \ref{fig:plots_transync_ls} that our methods perform better than the naive least-squares (LS) approach that simply estimates the strength vector individually on each graph $G_k$. In this case, one needs to impose connectivity on each graph in order to obtain meaningful results using LS. 
Figure \ref{fig:plots_transync_ls} shows that as expected, the MSE is constant with $T$ for the LS method. This illustrates that even when all the graphs are connected, dynamic approaches are better suited to recover the strengths and/or ranking of a set of items.

\item \rev{We show in Figure \ref{fig:sparsity} the influence of the sparsity of the input graphs on the estimation. We generate input graphs as $\calG(n,p)$ for different values of $p$ and observe that the MSE increases with the sparsity for the DLS method. The Projection method however has a similar performance for all sparsity levels.}

\item \rev{In Figure \ref{fig:perf}, we show that the optimal values for the hyperparameters $\lambda$ and $\tau$ derived theoretically are also numerically optimal. Indeed for both methods, the MSE is close to its minimum for these choices of parameters.}

\end{enumerate}


\begin{figure}[h!]
\centering
\begin{subfigure}{.5\textwidth}
  \centering
  \includegraphics[width=\linewidth]{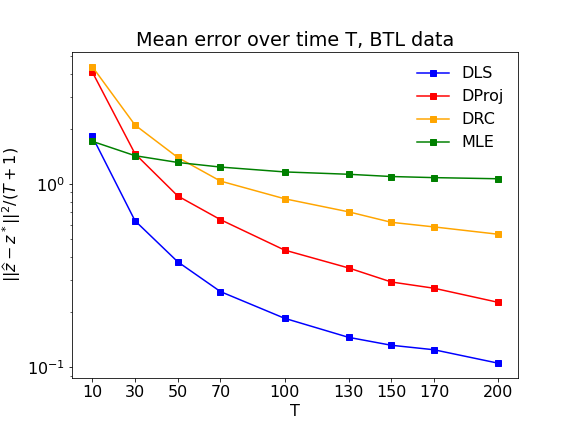}
  \caption{Smoothness $S_T = \frac{1}{T}$}
  \label{fig:plot_btl_lipschitz}
\end{subfigure}%
\begin{subfigure}{.5\textwidth}
  \centering
  \includegraphics[width=\linewidth]{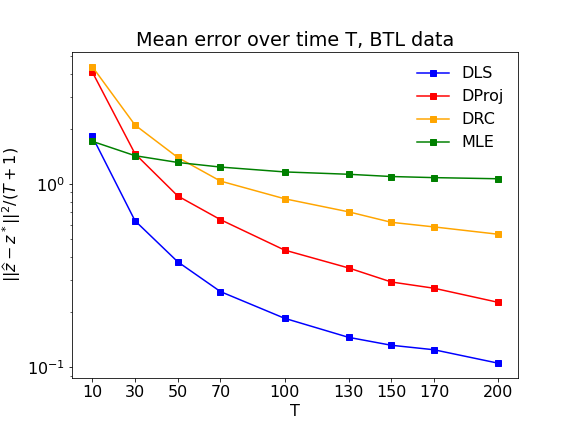}
  \caption{Smoothness $S_T = \frac{1}{\sqrt{T}}$}
\end{subfigure}%
\caption{\rev{MSE versus $T$ for DLS, DProj, DRC and MLE  when the data is generated according to the BTL model for $n=100$, and graphs are $G(n,p(t))$ with $p(t)$ chosen randomly between $\frac{1}{n}$ and $\frac{\log(n)}{n}$. The results are averaged over the grid $\calT$ as well as $40$ Monte Carlo runs.}}
\label{fig:plots_btl}
\end{figure}

%
\begin{figure}[h!]
\centering
\begin{subfigure}{.5\textwidth}
  \centering
  \includegraphics[width=\linewidth]{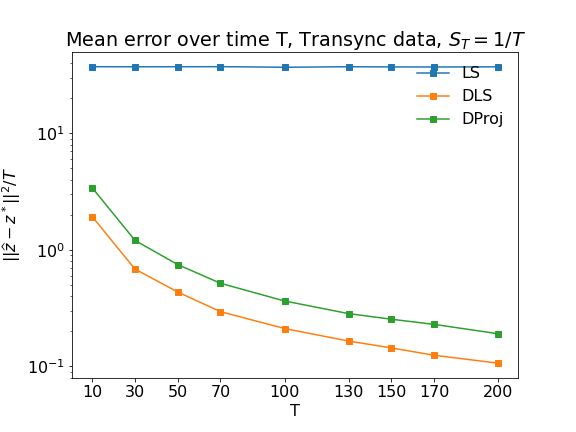}
  \caption{Smoothness $S_T = \frac{1}{T}$}
  \label{fig:transync_ls_alpha_ST_1}
\end{subfigure}%
\begin{subfigure}{.5\textwidth}
  \centering
  \includegraphics[width=\linewidth]{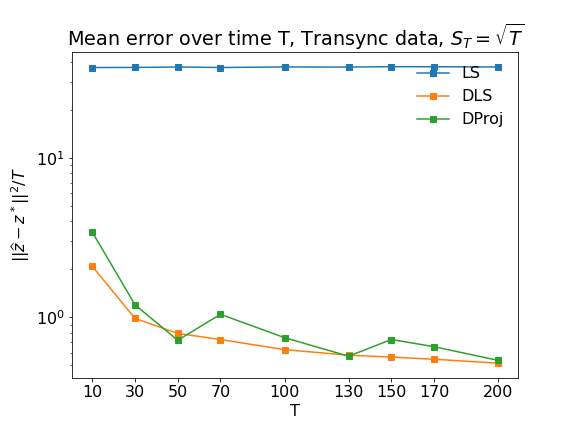}
  \caption{Smoothness $S_T = \sqrt{T}$}
\end{subfigure}
\caption{\rev{Evolution of estimation errors with $T$ for Least-Squares, DLS and DProj method when the synthetic data are generated according to the Dynamic TranSync model for $n=100$ and the graphs are generated as $G(n,p(t))$ with $p(t) = \frac{\log(n)}{n}$. In particular, we ensure that the individual graphs are all connected. The results are averaged over the grid $\calT$ as well as $20$ Monte Carlo runs.}}
\label{fig:plots_transync_ls}
\end{figure}

\begin{figure}[h!]
\centering
\begin{subfigure}{.5\textwidth}
  \centering
  \includegraphics[width=\linewidth]{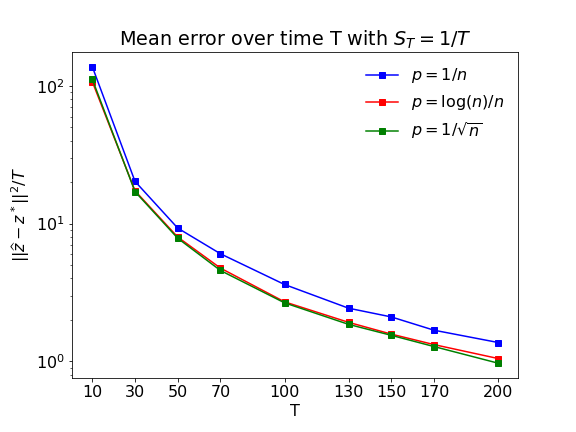}
  \caption{DLS Method}
  \label{fig:sparsity_ST_1_dls}
\end{subfigure}%
\begin{subfigure}{.5\textwidth}
  \centering
  \includegraphics[width=\linewidth]{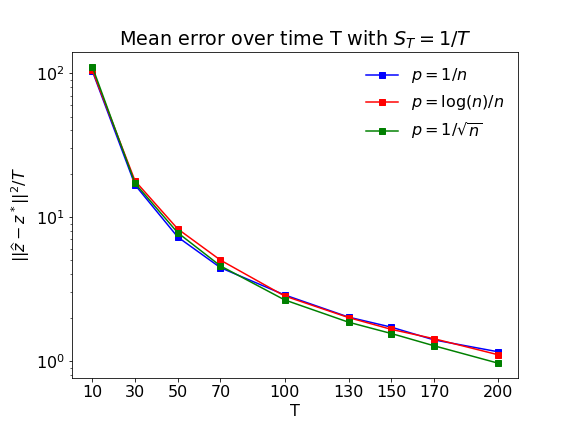}
  \caption{DProj Method}
    \label{fig:sparsity_ST_1_dproj}
\end{subfigure}
\caption{\rev{Evolution of estimation errors with $T$ for DLS and DProj method when the synthetic data are generated according to the Dynamic TranSync modelfor $n=100$ and the graphs are generated as $G(n,p)$ for different choices of $p$. In particular, we ensure that the individual graphs are all connected. The results are averaged over the grid $\calT$ as well as $20$ Monte Carlo runs.}}
\label{fig:sparsity}
\end{figure}

\begin{figure}[h!]
    \centering
    \begin{subfigure}{.5\textwidth}
        \centering
        \includegraphics[width=\linewidth]{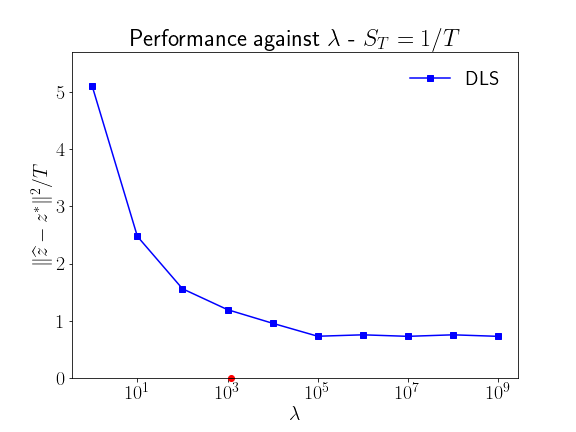}
        \label{fig:perf_lambda}
        \caption{Performance of DLS}
    \end{subfigure}%
    \begin{subfigure}{.5\textwidth}
        \centering
        \includegraphics[width=\linewidth]{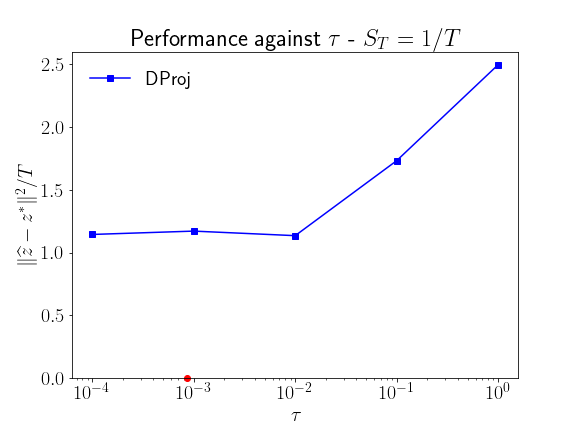}
        \label{fig:perf_tau}
        \caption{Performance of DProj}
    \end{subfigure}
    \caption{\rev{Performance of our methods for different values of hyperparameter, with $n = 100$, $T = 200$ and $\sigma=1$. We highlight in red on the x-axis the optimal values of parameter computed theoretically, $\lambda = \sigma^{4/3}\left(\frac{T}{S_T}\right)^{2/3}$ for the DLS and $\tau = \sigma^{-4/3}\left(\frac{S_T}{T}\right)^{2/3}$ for DProj.}}
    \label{fig:perf}
\end{figure}

%
\subsection{Real data}
We now provide empirical results on two real data sets -- the Netflix Prize data set \cite{jiang2011statistical}, and Premier League results from season $2000/2001$ to season $2017/2018$ \cite{ranksync21}.
In order to assess the performance of our algorithms, we will compute the number of upsets as well as the mean squared error\footnote{With a slight overload of notation, we use the term MSE here although it is not the same quantity as in \eqref{eq:mse_consis}.} (MSE) for each method.  The number of upsets is defined as the number of pairs for which the estimated preference is different from the observation. More precisely, denoting $\est z$ to be the estimator of the strengths (for any method),
\[ \text{Number of upsets} := \sum_{t\in\calT}\sum_{\set{i,j}\in\calE_t} \ones_{sign(y_{ij}(t)) \neq sign(\est z_{t,i} - \est z_{t,j})}.\]

The MSE is defined using the vector of observations $y$ and the estimated strength vector $\est z$ as
\[\text{MSE} = \frac{1}{T+1}\sum_{t\in\calT}\sum_{\set{i,j}\in\calE_t} \left(y_{ij}(t)-(\est z_{t,i}-\est z_{t,j})\right)^2. \]
The tuning parameters $\lambda, \tau$ will be chosen using the cross validation procedure described below.
\paragraph{Cross Validation Procedure.} 
\begin{enumerate}
    \item Fix a list of possible values of $\lambda$ (resp. $\tau$).
    \item For every possibles value of $\lambda$ (resp. $\tau$), repeat several times the steps below.
    \begin{itemize}
        \item For each time $t$, select randomly a measurement $y_{ij}(t)$ at time $t$. Let us denote $(i_t,j_t)$ the selected pair of items at time $t$. We denote $Y_{\text{test}} = \set{y_{i_t,j_t}(t) \, | \, t\in \calT}$ to be the set containing this data and $\calI_{\text{test}} = \set{(t,i_t,j_t) \, | \, t\in \calT}$ to be the set of corresponding indices.
        
        \item Consider the data set where the data in $Y_{\text{test}}$ have been removed. Compute the estimator $\est z$ on this smaller data set.
        
        \item Compute the prediction error for both performance criteria
        \begin{itemize}
            \item MSE: $\displaystyle\frac{1}{T+1}\sum_{(t,i_t,j_t)\in \calI_{\text{test}}}(y_{i_t,j_t} - (\est z_{i_t}(t) - \est z_{j_t}(t)))^2$.
            
            \item Mean number of upsets : $\displaystyle\frac{1}{T+1}\sum_{(t,i_t,j_t)\in \calI_{\text{test}}}\ones_{sign(y_{i_tj_t}(t)) \neq sign(\est z_{t,i_t} - \est z_{t,j_t})}$.
        \end{itemize}
    \end{itemize}
    We compute the mean of those errors for each value of the parameter $\lambda$ (resp. $\tau$).
    \item Select $\lambda^\star$ (resp. $\tau^\star$) which minimizes the mean prediction error.
    \item Proceed to the estimation with the chosen parameter $\lambda^\star$ (resp. $\tau^\star)$.
\end{enumerate}

\paragraph{Netflix Prize data set.} Netflix has provided a data set containing anonymous customer's ratings of 17770 movies between November 1999 and December 2005. These ratings are on a scale from 1 to 5 stars. From those individual rankings, we need to form pairwise information that satisfies the Dynamic Transync model \eqref{eq:transync_model}. Denoting $s_{i}(t)$ to be the mean score of movie $i$ at time $t$, computed as the mean rating given to this movie at time $t$ among the customers, we then define for all pair of movies $\set{i,j}$ rated at time $t$
\begin{equation}\label{eq:data_netflix}
    y_{ij}(t) = s_i(t)-s_j(t).
\end{equation}
For computational reasons, we choose a subset of 100 movies to rank. We can then use our estimation methods to recover each movie's quality (and then their rank) at any month between November 1999 and December 2005. In order to denoise the observations, we gather the data corresponding to successive months such that all the graphs of merged observations are connected. The merged dataset is then composed of $T = 23$ observation graphs. The associated observations $y_{ij}(t)$ are computed in a similar manner as \eqref{eq:data_netflix}, using the mean score of the movies across the corresponding merged time points and the customers.
%
%
%
%
\begin{table}
\begin{subtable}[c]{0.5\textwidth}
\centering
\begin{tabular}{c|c|c}
 & MSE & Upsets \\
\hline
$\lambda^*$ & 131.48 & 157.89\\
$\tau^* $ & 0.42 & 0.37
\end{tabular}
\subcaption{Cross-validation results}
\label{tab:cv_results_netflix}
\end{subtable}\hspace{1cm}
\begin{subtable}[c]{0.5\textwidth}
\centering
\begin{tabular}{c|c|c}
 & MSE & Upsets \\
\hline
LS & 3573 & 0.4954\\
DLS & 3040 & 0.4984 \\
DProj & 3025 & 0.4976
\end{tabular}
\subcaption{Performance of the three estimation methods}
\label{tab:results_netflix}
\end{subtable}
\caption{Cross-validation and performance for the chosen parameters for the Netflix dataset. We use the MSE and the mean number of Upsets as our performance criteria.}
\end{table}
For this dataset of merged movies, the cross validation step gives two different values of parameters $(\lambda^*,\tau^*)$ for the two different performance criteria we consider, presented in Table \ref{tab:cv_results_netflix}. Using $\lambda^*, \tau^*$, we then perform the estimation using the naive Least-Squares (LS) approach as well as our two methods, namely DLS and DProj, and compute in both cases the MSE and the mean number of upsets. The results are presented in Table \ref{tab:results_netflix}. We note that the mean number of upsets is essentially similar for all the methods, indicating that none of the methods is better than the other for this criteria. However, DLS and DProj improve the performance of the estimation in terms of the MSE criterion.

\paragraph{English Premier League dataset.} This dataset is composed of all the game results from the English Premier League from season 2000-2001 to season 2017-2018. These 18 seasons involve $n=43$ teams in total, each season seeing $20$ teams confront each other across $38$ rounds. The observations are the mean scores of the games between a pair of teams within the same season. Similar to the Netflix dataset, we can group the game results from successive seasons, resulting in $T = 9$ observation graphs. Let us denote $\calT_k$ to be the set of time-points gathered to form the graph $G_k$. For each $G_k$, the corresponding observations are defined for each pair of teams $\set{i,j} \in \calE_k$ as
\[y_{ij}(k) = \frac{1}{|\calT_k|}\sum_{t\in\calT_k} (s_i(t)-s_j(t)),\]
where $s_i(t)$ denotes the mean number of goals scored by team $i$ against team $j$ during the season $t$. In this case, merging the data does not lead to individual connectivity of the graphs, because of the promotion and relegation of teams at the end of each season. However, the union of all these graphs is connected. As for the Netflix dataset, we perform a cross-validation step to choose the best parameters $(\lambda^*,\tau^*)$, in regards to the performance criterion we consider (MSE or Mean number of upsets). The results of the cross validation are presented in Table \ref{tab:cv_results_epl}. Then, we perform our estimation for these chosen values of the parameters for DLS and DProj. Note that as the individual graphs are not connected, the naive LS method will not give interpretable results,and is only included here for comparison with our methods. As shown in Table \ref{tab:results_epl}, DLS and DProj perform better than LS both in terms of MSE or in Mean number of upsets. Specifically, the number of upsets is improved by 10$\%$ with our methods.


\begin{table}
\begin{subtable}[c]{0.5\textwidth}
\centering
\begin{tabular}{c|c|c}
& MSE & Upsets \\
\hline
$\lambda^*$ & 24.49 & 32.65\\
$\tau^* $ & 6.43 & 5.57
\end{tabular}
\subcaption{Cross-validation results}
\label{tab:cv_results_epl}
\end{subtable}\hspace{1cm}
\begin{subtable}[c]{0.5\textwidth}
\centering
\begin{tabular}{c|c|c}
 & MSE & Upsets \\
\hline
LS & 0.0024 & 0.67\\
DLS & 0.0015 & 0.57 \\
DProj & 0.0014 & 0.58
\end{tabular}
\subcaption{Performance of the three estimation methods}
\label{tab:results_epl}
\end{subtable}
\caption{Cross-validation and performances for those chosen parameters for the Premier League dataset. We use the MSE and the mean number of Upsets as our performance criteria. Performance results are presented for the simple LS, DLS and DProj methods.}
\end{table}

\rev{\paragraph{Smoothness of the real data sets.} Our methods rely on the underlying smoothness of the data. However, with real data sets, the ground truth vector $z^*$ is unknown, thus making it difficult to check whether this assumption is satisfied. In order to verify that our data sets are  fit for our methods, one can define, as a proxy, an ``empirical ground-truth'' vector 
\[
z^{*,emp}_{t,i} := \frac{1}{\abs{\calN_{t,i}}} \sum_{j \in \calN_{t,i}} y_{ij}(t),
\]
where $\calN_{t,i}$ denotes the set of neighbors of node $i$ in the graph $G_t$. For some items that have been compared at all times, we plot the evolution of $z^{*,emp}_{t,i}$ in Figure \ref{fig:smoothness_real_data}. This figure shows that despite some jumps, the overall evolution of $z^{*,emp}_{t,i}$ is reasonably smooth.}
\begin{figure}[h!]
    \centering
    \begin{subfigure}{.5\textwidth}
        \centering
        \includegraphics[width=\linewidth]{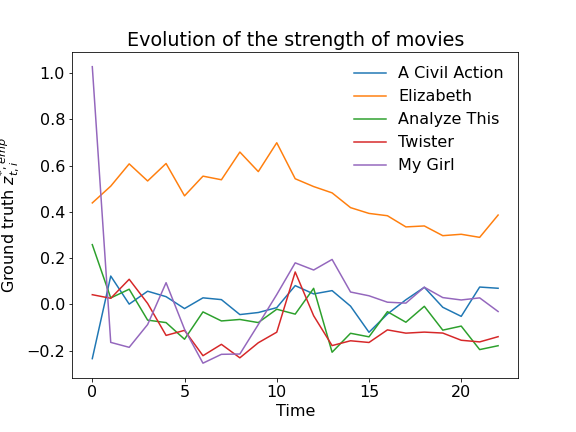}
    \end{subfigure}%
    \begin{subfigure}{.5\textwidth}
        \centering
        \includegraphics[width=\linewidth]{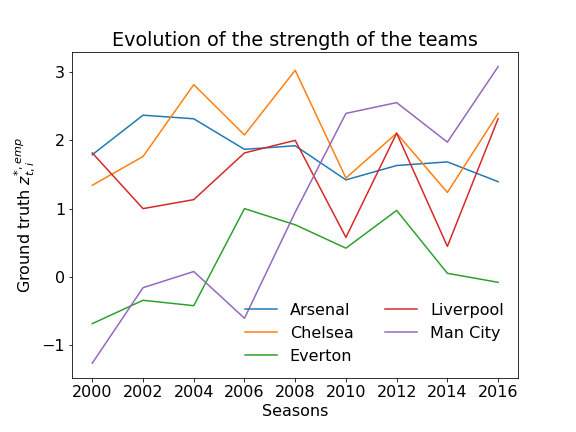}
    \end{subfigure}
    \caption{Evolution of the strengths for real data sets }
    \label{fig:smoothness_real_data}
\end{figure}

\section{Concluding remarks}\label{sec:conclusion}
In this paper we considered the problem of estimating the latent strengths of a set of $n$ items from noisy pairwise measurements in a dynamic setting. In particular, we proposed a dynamic version of the TranSync model \cite{huang2017translation} by placing a global smoothness assumption on the evolution of the latent strengths. We proposed and analyzed two estimators for this problem and obtained $\ell_2$ estimation error rates for the same. Experiment results on both synthetic data and real data sets were presented.

Some interesting directions for future work are as follows.
\begin{enumerate}
    \item \textit{Disconnected graphs.} As previously discussed in Remarks \ref{rem:connectedness_assum2} and \ref{rem:connectedness_assum}, we believe that our analysis can be extended to the case where some of the individual graphs are disconnected. It will be interesting to analyze this setup in detail with corresponding error bounds.
    
    \item \textit{Lower bound.} Our analysis focuses on upper bounds on the estimation error, however it will be interesting to derive lower bounds for the dynamic TranSync model which showcase the optimal dependence on $n, T$ and $S_T$.  
    
    \item \textit{Dynamic BTL model.} Finally, an interesting direction would be to theoretically analyze the performance of the proposed estimators for the dynamic BTL model (recall Remark \ref{rem:btl_setup}). The main difficulty in this regard is that the noise in the measurements is not zero-mean anymore, and it is not easy to see how such a noise term can be handled to give meaningful error bounds.

\end{enumerate}


\clearpage
\bibliographystyle{abbrvnat}
\bibliography{references}

\newpage
\appendix
\section{Proof of Lemma \ref{lem:rank_Laplacian_reg}} \label{appsubsec:rank_Laplacian_reg}
\begin{proof}
 Recall that $(\lambda_i,v_i)_{i=1}^n$ denote the eigenpairs of $\complincmat\complincmat^\top$ with $\lambda_1 \geq \cdots \geq \lambda_{n-1} > \lambda_n$ its eigenvalues. Since $\complincmat\complincmat^\top = nI - \ones\ones^\top$ we know that $v_n = \ones_n$ and $\set{v_i}_{i=1}^{n-1}$ is any orthonormal basis for the space orthogonal to span$(\ones_n)$. Also recall that $(\mu_k,u_k)_{k=0}^{T}$ denote the eigenpairs of $M M^\top$ where $u_{T} = \ones_{T+1}$. Now the eigenvectors of $E^\top E$ that have eigenvalue zero are 
 \begin{itemize}
     \item $u_T \otimes v_j = \ones_{T+1} \otimes v_j$ for $j=1,\dots,n-1$, and
     
     \item $u_k \otimes v_n = u_k \otimes \ones_n$ for $k=0,\dots,T$.
 \end{itemize}

 Note that $\set{u_k \otimes \ones_n}_{k=0}^{T}$ lie in $\calN(L+\lambda E^\top E)$ and
 \begin{equation*}
  \text{span}\set{u_k \otimes \ones_n}_{k=0}^{T} = \text{span}\set{e_k \otimes \ones_n}_{k=0}^{T}.
 \end{equation*}
 Since $L+\lambda E^\top E$ is p.s.d, we have that $\calN(L+\lambda E^\top E) = $ \text{span}$\set{e_k \otimes \ones_n}_{k=0}^{T}$ iff 
 \begin{equation} \label{eq:temp1}
     x^\top(L_{\tilde V}+\lambda E^\top E)x > 0, \quad \forall x \ ( \neq 0) \in  \text{span}^{\perp}\set{e_k \otimes \ones_n}_{k=0}^{T}.
 \end{equation}
 As the orthogonal complement of $\text{span}\set{e_k \otimes \ones_n}_{k=0}^{T}$ is given by 
 \begin{equation*}
\text{span}^{\perp}\set{e_k \otimes \ones_n}_{k=0}^{T} = \text{span}\set{\ones_{T+1} \otimes v_j}_{j=1}^{n-1} \oplus \calN^{\perp}(E^\top E) 
 \end{equation*}
we claim that \eqref{eq:temp1} translates to establishing that
 \begin{equation} \label{eq:temp2} 
      x^\top L x > 0, \quad \forall x \ (\neq 0) \in  \text{span}\set{\ones_{T+1} \otimes v_j}_{j=1}^{n-1}.
 \end{equation}
To prove this claim, we begin by writing any $x \in \text{span}^{\perp}\set{e_k \otimes \ones_n}_{k=0}^{T}$ as $x = \tilde x + x'$ where $\tilde x \in \text{span}\set{\ones_{T+1} \otimes v_j}_{j=1}^{n-1}$ and $x' \in \calN^{\perp}(E^\top E)$. Then if $x \neq 0$, 
\begin{align*}
x^\top(L+\lambda E^\top E)x 
&= x^\top Lx + \lambda x'^\top E^\top E x' \\ 
&= \left\{
\begin{array}{rl}
\tilde x^\top L \tilde x  \ ; & \text{ if } x' = 0, \\
> 0 \ (\text{since } \lambda > 0)\ ; & \text{ if } x' \neq 0,
\end{array} \right.
\end{align*}
which establishes the claim since at least one of $\tilde x, x' \neq 0$.

To prove \eqref{eq:temp2},  we first observe that $x^\top (L+\lambda E^\top E) x = x^\top L x$, for any $x \in \text{span}\set{\ones_{T+1} \otimes v_j}_{j=1}^{n-1}$. Since $v_1,\dots,v_{n-1}$ is any orthonormal basis for the subspace orthgonal to $\ones_n$, we set $x = \ones_{T+1} \otimes v$ for any $v (\neq 0)$ lying in that subspace. This gives us
\begin{equation*}
    x^\top L x = \sum_{t \in \calT} v^\top L_t v = v^\top \left(\sum_{t \in \calT} L_t \right) v.
\end{equation*}
Now, $\sum_{t \in \calT} L_t$ is the Laplacian of the union graph $G_U$. Since $\sum_{t \in \calT} L_t$ has rank $n-1$ iff $G_U$ is connected (which is true by assumption), we arrive at \eqref{eq:temp2}. 
\end{proof}

\section{Useful technical results}
\rev{To make our paper self-contained, we recall below two useful and standard results regarding the L\"owner ordering. The proof of the next lemma has been adapted from \cite{mathstack1}.}
\rev{
\begin{lemma}\label{lem:loewner_inverse}
Let $A,B$ be two $N \times N$ symmetric positive definite matrices, such that $B\succcurlyeq A$. Then, it holds $A^{-1}\succcurlyeq B^{-1}$. 
\end{lemma}
}
\begin{proof}
\rev{
    Given that $B\succcurlyeq A$, it is easy to see that $A^{-1/2}(B-A)A^{-1/2}\succcurlyeq 0$, which is equivalent to $A^{-1/2}BA^{-1/2}\succcurlyeq I$. Assuming $A$ and $B$ are $N\times N$ matrices, let us write the spectral expansion 
    \[A^{-1/2}BA^{-1/2}=\sum^N_{i=1}\mu_iu_iu_i^\top.\]  
    Noticing that $I=\sum^N_{i=1}u_iu_i^\top$, and given the fact $A^{-1/2}BA^{-1/2}\succcurlyeq I$ it is clear that $\mu_i>1$ for all $i$, which implies that \[I=\sum^N_{i=1}u_iu_i^\top\succcurlyeq \sum^N_{i=1}\mu^{-1}_iu_iu_i^\top= A^{1/2}B^{-1}A^{1/2}.\]
    On other hand, using this Loewner inequality, we deduce that \begin{align*}
        B^{-1}=A^{-1/2}(A^{1/2}B^{-1}A^{1/2})A^{-1/2} 
        \preccurlyeq A^{-1/2}IA^{-1/2}
        =A^{-1}
    \end{align*}
    }
\end{proof}
\rev{The proof of the following lemma is outlined in \cite{mathstack2}.
\begin{lemma}\label{lem:loewner_pseudo}
    Let $A,B$ be two $N \times N$ symmetric positive semidefinite matrices satisfying $B\succcurlyeq A$ and $\calN(A)=\calN(B)$. Then it holds $A^\dagger \succcurlyeq B^\dagger$.
\end{lemma}
}
\begin{proof}
\rev{
If $A,B$ are invertible, the result follows direclty from Lemma \ref{lem:loewner_inverse}. If they are not, assume that their nullspace $\calN(A)$ has dimension $d \geq 1$ and their range has dimension $N-d$. Let $\{v_1,\cdots,v_d\}$ be any orthonormal basis of $\calN(A)$. It is easy to see that $A+\sum^d_{i=1}v_iv^T_i$ is invertible, and the same is true for $B+\sum^d_{i=1}v_iv^T_i$. Given $B+\sum^d_{i=1}v_iv^T_i \succcurlyeq A+\sum^d_{i=1}v_iv^T_i$, we use Lemma \ref{lem:loewner_inverse}, to deduce 
\begin{align*}
A^{\dagger}+\sum^d_{i=1}v_iv^T_i = (A+\sum^d_{i=1}v_iv^T_i)^{-1}
    \succcurlyeq (B+\sum^d_{i=1}v_iv^T_i)^{-1} 
    = B^\dagger+\sum^d_{i=1}v_iv^T_i,
\end{align*}
where for the first and last equality we used the orthogonality of the range of $A$ ($B$ has the same range) with respect to the nullspace. From this, the result follows. 
}
\end{proof}

\section{\rev{Additional simulations}}
\paragraph{Eigenvalues of $L(\lambda)$.} \rev{The theoretical analysis of the smoothness penalized estimator relies on the knowledge of the eigenvalues of $L(\lambda)$ and how they differ from the eigenvalues of $L$. We show in Figure \ref{fig:eigs} the variations of the eigenvalues of these matrices for different choices of $\lambda$. We verify experimentally that the addition of a penalty term increases part of the spectrum, which is at the base of our theoretical analysis.
}
\begin{figure}[h!]
    \centering
    \begin{subfigure}{.5\textwidth}
        \centering
        \includegraphics[width=\linewidth]{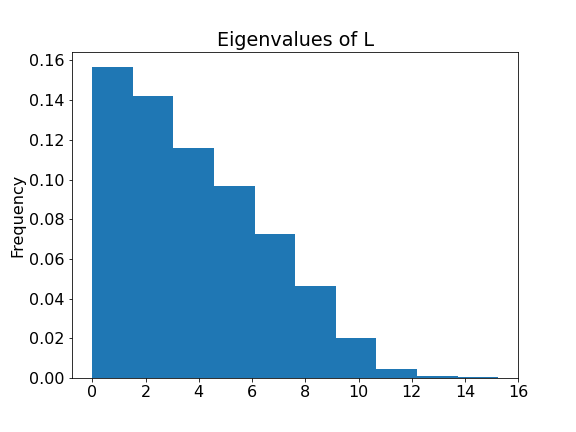}
        \label{fig:eigs_L}
        \caption{Eigenvalues of $L$}
    \end{subfigure}%
    \begin{subfigure}{.5\textwidth}
        \centering
        \includegraphics[width=\linewidth]{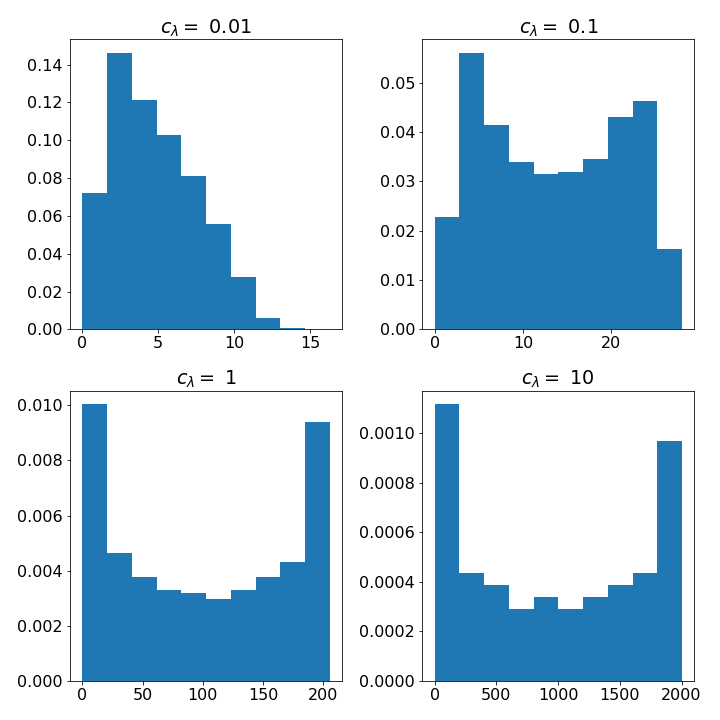}
        \label{fig:eigs_L_lam}
        \caption{Eigenvalues of $L(\lambda)$}
    \end{subfigure}
    \caption{\rev{Histogram of eigenvalues of $L$ and $L(\lambda)$ for different values of $\lambda = c_{\lambda}\left(\frac{T}{S_T}\right)^{2/3}$, for $T = 200$ and $n = 100$. The addition of the penalty term increases the eigenvalues.}}
    \label{fig:eigs}
\end{figure}

\paragraph{Simulation for $n \gg T$.} \rev{Figure \ref{fig:plots_transync_N200} shows the evolution of the MSE for $n = 200$ and $T$ varying between $10$ and $50$. We observe that even in the case $n \gg T$, the estimation error goes to 0 as $T$ increases. However, the errors are slightly bigger than in the case $n \ll T$, which is consistent with the theoretical error bounds we presented in this paper.}

\begin{figure}[h!]
\centering
\begin{subfigure}{.5\textwidth}
  \centering
  \includegraphics[width=\linewidth]{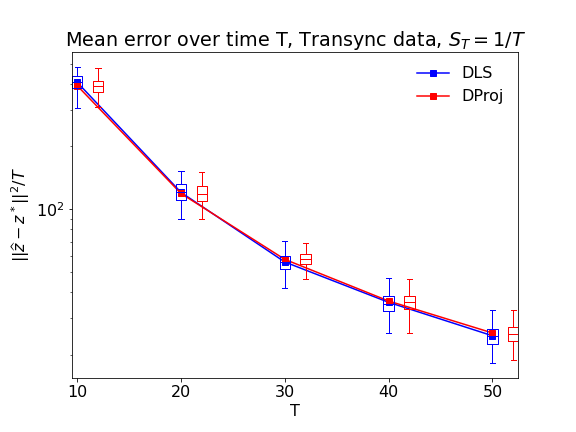}
  \caption{Smoothness $S_T = \frac{1}{T}$}
  \label{fig:transync_alpha_ST_1_N200}
\end{subfigure}%
\begin{subfigure}{.5\textwidth}
  \centering
  \includegraphics[width=\linewidth]{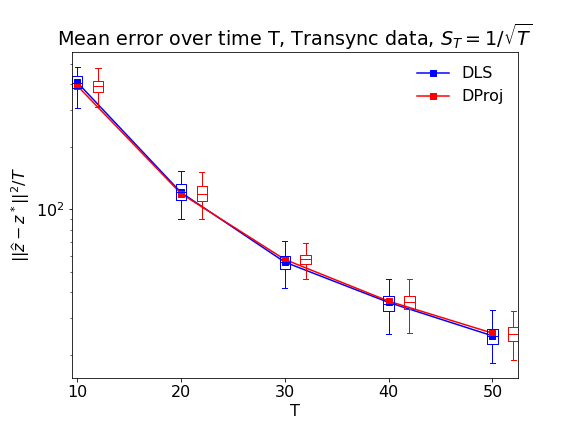}
  \caption{Smoothness $S_T = \frac{1}{\sqrt{T}}$}
  \label{fig:transync_alpha_ST_05_N200}
\end{subfigure}%
\hfill
\caption{\rev{MSE versus $T$ for DLS and DProj for $n \gg T$. Here $n = 200$, $T$ goes from $10$ to $50$ and data is generated according to the Dynamic TranSync model. Graphs are generated as $G(n,p(t))$ with $p(t)$ chosen randomly between $\frac{1}{n}$ and $\frac{\log(n)}{n}$. The results are averaged over the grid $\calT$ as well as $20$ Monte Carlo runs.}}
\label{fig:plots_transync_N200}
\end{figure}

\end{document}